\documentclass[11pt,reqno]{amsart}

\usepackage{amsmath,amsfonts,amssymb,amscd,amsthm,amsbsy,bm,latexsym,enumerate,mathrsfs,cases}
\usepackage{gensymb}
\usepackage[pdftex,bookmarksnumbered=true,bookmarksopen=true,          
            colorlinks=true,pdfborder=001,citecolor=blue,
            linkcolor=red,anchorcolor=green,urlcolor=blue]{hyperref}
\usepackage{graphicx}
\allowdisplaybreaks

\usepackage{color}

\theoremstyle{definition}
\newtheorem{definition}{Definition}
\theoremstyle{plain}
\newtheorem{thm}{Theorem}[section]
\newtheorem{theorem}{Theorem}

\newtheorem{lemma}{Lemma}

\theoremstyle{definition}
\newtheorem{remark}{Remark}
\numberwithin{equation}{section}
\title[Principal eigenvalue of an elliptic operator with large advection]{Asymptotics of the principal eigenvalue of a linear elliptic operator with large advection}

\author[Rui Peng]{Rui Peng${}^{\dag}$}
\author[Guanghui Zhang]{Guanghui Zhang${}^{\S,*}$}

\thanks{${}^\dag$School of Mathematical Sciences, Zhejiang Normal University, Jinhua, Zhejiang, 321004, China. (Email: {\tt pengrui\,$\b{}$\,seu@163.com})}
\thanks{${}^\S$School of Mathematics and Statistics, Huazhong University of Science and Technology, Wuhan,
430074, China. (Email: {\tt guanghuizhang@hust.edu.cn})}
\thanks{${}^*$Hubei Key Laboratory of Engineering Modeling and Scientific Computing, Huazhong University of Science and Technology, Wuhan 430074, China.}


\begin{document}

\begin{abstract}
Consider the eigenvalue problem of a linear second order elliptic operator:
 \begin{equation}
  \nonumber
  -D\Delta \varphi -2\alpha\nabla m(x)\cdot \nabla\varphi+V(x)\varphi=\lambda\varphi\ \ \hbox{ in }\Omega,
\end{equation}
complemented by the Dirichlet boundary condition or the following general Robin boundary condition:
 $$
 \frac{\partial\varphi}{\partial n}+\beta(x)\varphi=0 \ \ \hbox{ on  }\partial\Omega,
 $$
where $\Omega\subset\mathbb{R}^N (N\geq1)$ is a bounded smooth domain, $n(x)$ is the unit exterior normal to
$\partial\Omega$ at $x\in\partial\Omega$, $D>0$ and $\alpha>0$ are, respectively, the diffusion and advection coefficients, $m\in C^2(\overline\Omega),\,V\in C(\overline\Omega)$,  $\beta\in C(\partial\Omega)$ are given functions, and $\beta$ allows to be positive, sign-changing or negative.

In this paper, we aim to establish, as $\alpha$ approaches $\infty$, the asymptotic behavior of the principal eigenvalue under appropriate conditions on the advection function $m$. For $N=1$, we provide a complete characterization of the asymptotic behavior, assuming that the derivative of $m$ changes sign at most finitely many times. Our findings not only improve upon the previous work in \cite{BHN2005,CL2008,PZ2018}, but also partially address some of the open questions posed in \cite{BHN2005}. Furthermore, our results elucidate the novel influence of boundary conditions on such asymptotics.

\end{abstract}

\maketitle

\section{Introduction}
In this paper, we consider the following eigenvalue problem of a linear second order elliptic operator with Robin boundary condition:
\begin{equation}
  \label{eq:1.1}
  \left\{
    \begin{aligned}
      &-D\Delta \varphi-2\alpha\nabla m(x)\cdot\nabla\varphi+
      V(x)\varphi=\lambda\varphi &&\text{in }\Omega,\\
      &\frac{\partial\varphi}{\partial n}+\beta(x)\varphi=0 &&\text{on }\partial\Omega,
    \end{aligned}
    \right.
\end{equation}
where $\Omega\subset\mathbb{R}^{N} (N\geq1)$ is a bounded smooth domain, $n(x)$ is the unit exterior normal to
$\partial\Omega$ at $x\in\partial\Omega$, $V\in C(\overline{\Omega})$, $\beta\in
C(\partial\Omega)$, the advection function $m\in C^{2}(\overline{\Omega})$, and $D>0$ and $\alpha>0$ stand for, respectively, the diffusion and advection coefficients. It is worth mentioning that the function $\beta$ on the boundary condition allows to be positive, sign-changing or negative in our study.

We are also concerned with the eigenvalue problem with Dirichlet boundary condition:
\begin{equation}
  \label{eq:1.1-d}
  \left\{
    \begin{aligned}
      &-D\Delta \varphi-2\alpha\nabla m(x)\cdot\nabla \varphi+
      V(x)\varphi=\lambda\varphi &&\text{in }\Omega,\\
      &\varphi=0 &&\text{on }\partial\Omega
    \end{aligned}
  \right.
  \end{equation}
with $m,\,V$ and $\Omega$ fulfilling the same smoothness assumptions as before.

Given $D>0,\,\alpha,\,m$ and $V$, it is well known that \eqref{eq:1.1} (and \eqref{eq:1.1-d}) admits a smallest eigenvalue (also called as {\it principal eigenvalue}), denoted by $\lambda_1(\alpha)$, which corresponds to a positive eigenfunction (called as {\it principal eigenfunction}). For the existence and uniqueness of the principal eigenvalue of the eigenvalue problems \eqref{eq:1.1},  \eqref{eq:1.1-d} and \eqref{p} in Section 4, one may refer to \cite[Theorem 12.1]{Am83}, \cite[Theorem 2.2]{AmL98} or \cite{Du, FK2015, PZ2018}. The variational characterization of the principal eigenvalues of \eqref{eq:1.1} (i.e., \eqref{eq:1.2} below) and \eqref{p} (i.e., \eqref{4.1e} below), which will be used frequently later, is also a standard result. This can be easily proved similarly to the classical case of $\beta\geq0$ involving the use of the trace theorem (see \cite[Theorem 1.5.1.10]{Gr} or \cite[Lemma 3.1]{PZZ2019}).

The principal eigenvalue is a basic concept in the area of partial differential equations, and it has found many important applications in various kinds of nonlinear modelling problems arising from biology, chemical reactions, material sciences, etc; one may refer to \cite{ALL2017,CC2003,CL2008,CLPZ,CL,CLL,Du,DH2008_1,DH2008_2,DP2012,GN2013,KMP,
LL2019_1,LLou,LPZ2019,LLZ2022,PZ2015,PZ2016,ZZ2016,ZX2018} and the references therein.

The main focus of our work is to demonstrate how boundary conditions play a crucial role in the asymptotic behavior of the principal eigenvalue, particularly in the context of large advection rate or large/small diffusion rate. In \cite{PZZ2019}, we investigated the asymptotic behavior of the principal eigenvalue $\lambda_1(\alpha)$ of \eqref{eq:1.1} and \eqref{eq:1.1-d} as the diffusion coefficient $D$ goes to $0$ or $\infty$. In this paper, we aim to explore the asymptotic behavior of $\lambda_1(\alpha)$ as the advection coefficient $\alpha$ approaches $\infty$. 

Let us first recall some existing research works on the asymptotic behavior of the principal eigenvalue with respect to large advection closely related to \eqref{eq:1.1} and \eqref{eq:1.1-d}. In \cite{BHN2005}, Berestycki, Hamel and Nadirashvili considered the following eigenvalue problem subject to Neumnn boundary condition:
\begin{equation}\label{N-eig}
  \left\{
    \begin{aligned}
      &-D\Delta \varphi-2\alpha v(x)\cdot\nabla \varphi+V(x)\varphi=\lambda\varphi &&\text{in }\Omega,\\
      &\frac{\partial\varphi}{\partial n}=0 &&\text{on }\partial\Omega
    \end{aligned}
    \right.
\end{equation}
as well as the one subject to Dirichlet boundary condition:
 \begin{equation}\label{D-eig}
  \left\{
    \begin{aligned}
      &-D\Delta \varphi-2\alpha v(x)\cdot\nabla \varphi+V(x)\varphi=\lambda\varphi &&\text{in }\Omega,\\
      &\varphi=0 &&\text{on }\partial\Omega,
    \end{aligned}
    \right.
\end{equation}
where $v\in L^\infty(\Omega)$ is a divergence-free vector field;
that is,
 \begin{equation}\label{div-f}
{\rm div}\,v=0\ \  \, \mbox{ in}\ \mathcal{D}'(\Omega).
\end{equation}
Actually, \cite{BHN2005} dealt with a bit more general elliptic operator and also the associated space periodic eigenvalue problem. Note that if $v\in C^1(\overline\Omega)$ is a gradient field (i.e, $v=\nabla m$), then \eqref{N-eig} and \eqref{D-eig} reduce, respectively, to \eqref{eq:1.1} with $\beta=0$ and \eqref{eq:1.1-d}.

We now recall the following definition introduced in \cite{BHN2005}.
\begin{definition}\label{def-0}

\vskip5pt \item {\rm (i)} A function $u$ is said to be a {\bf first integral} of the vector field $v$ if
$u\in H^1(\Omega)$ and $v\cdot\nabla u=0$ almost everywhere in $\Omega$.

\vskip5pt
\item {\rm (ii)} Denote $\mathcal{I}$ (resp. $\mathcal{I}_0$) to be the set of all first integrals of $v$ which belong to $H^1(\Omega)$ (resp. $H^1_0(\Omega)$).

\end{definition}

Berestycki, Hamel and Nadirashvili derived the following important result for both \eqref{N-eig} and \eqref{D-eig};
see \cite[Theorems 2.1 and 2.2]{BHN2005}.

\begin{thm}\label{theorem-ND} Assume that \eqref{div-f} holds. The following assertions hold.

\item {\rm(i)} Let $\lambda_1(\alpha)$ be the principal eigenvalue of \eqref{N-eig} and assume that $v\cdot n=0$ in $L_{loc}^1(\partial\Omega)$. Then $\lambda_1(\alpha)$ is bounded and
  \[
    \lim_{\alpha\to\infty}\lambda_1(\alpha)=\min_{u\in \mathcal{I}}\frac{\int_\Omega{D|\nabla u|^2+V(x)u^2}}{\int_\Omega u^2}.
    \]

\item {\rm(ii)} Let $\lambda_1(\alpha)$ be the principal eigenvalue of \eqref{D-eig}. Then $\lim_{\alpha\to\infty}\lambda_1(\alpha)=\infty$ if and only if $v$ has no first integral in $H^1_0(\Omega)$, and if $v$ has a first integral in $H^1_0(\Omega)$, then
    \[
    \lim_{\alpha\to\infty}\lambda_1(\alpha)=\min_{u\in \mathcal{I}_0}\frac{\int_\Omega{D|\nabla u|^2+V(x)u^2}}{\int_\Omega u^2}.
    \]
\end{thm}

As mentioned in \cite[Reamrk 2.5]{BHN2005}, in Theorem \ref{theorem-ND}(i), the extra assumption that $v\cdot n=0$ in $L_{loc}^1(\partial\Omega)$ is necessary in order to determine the limit $\lim_{\alpha\to\infty}\lambda_1(\alpha)$ of problem \eqref{N-eig}. On the other hand, for problem \eqref{D-eig}, Devinatz,  Ellis and Friedman \cite{DEF1974}, Friedman \cite{F1973} and Wentzell \cite{W1975} derived several nice results on the limiting behavior (even the asymptotic rate) of the principal eigenvalue as $\alpha\to\infty$; see \cite{DEF1974,F1973,W1975} for more details. Also pointed out by \cite[Reamrk 1.1]{BHN2005}, if the vector field $v$ is not divergence-free, Theorem \ref{theorem-ND}(ii) does not hold; indeed, \cite{F1973} showed that if $V=0$, $v$ is continuous on $\overline\Omega$ and $v\cdot n<0$ on $\partial\Omega$, then $\lambda_1(\alpha)=O(\alpha e^{-c_0\alpha})$ for some $c_0>0$ as $\alpha\to\infty$. Observe that Theorem \ref{theorem-ND}(ii) claims that $\lim_{\alpha\to\infty}\lambda_1(\alpha)$ is $+\infty$ or a positive constant when $V=0$ and $v$ is divergence-free. The main results of \cite{BHN2005} were extended in \cite{GGP2010,GGP2013} to the elliptic operator with an indefinite weight function in front of $\lambda$ in \eqref{eq:1.1-d}.

Recently, in \cite{CL2008}, as $\alpha\to\infty$, Chen and Lou studied the asymptotic behavior of $\lambda_1(\alpha)$ of the following eigenvalue problem with a gradient field and Neumann boundary condition (i.e., $v=\nabla m$ in \eqref{N-eig}):
\begin{equation}\label{eq:N}
\left\{
\begin{aligned}
  &-\Delta \varphi -2\alpha\nabla m(x)\cdot \nabla\varphi+V(x)\varphi=\lambda\varphi &&\hbox{ in }\Omega,\\
  &\frac{\partial\varphi}{\partial n}=0 &&\hbox{ on }\partial\Omega.
\end{aligned}
\right.
\end{equation}

In the following, we shall introduce the main result in \cite{CL2008}. To do so, we need to recall the definitions of critical points and local maxima for the advection function $m$ used in \cite{CL2008}.

\begin{definition}\label{def-1}
  Assume that $m\in C^{2}(\overline{\Omega})$.

\vskip5pt \item {\rm (i)} A point $x\in\Omega$ is called an \textbf{interior critical
    point} of $m$ if $|\nabla m(x)|=0$. An interior critical point $x$
  is called \textbf{non-degenerate} if $\text{det}(D^{2}m(x))\neq 0$.

\vskip5pt
\item {\rm (ii)} A point $x\in\partial\Omega$ is called a \textbf{boundary critical
    point} of $m$ if $|\nabla m(x)|=|n(x)\cdot\nabla m(x)|$. A boundary critical point $x$
  is called \textbf{non-degenerate} if either $|\nabla m(x)|=0$, $
  \text{det}\left(D^{2}m(x)\right)\neq 0$ or $|\nabla m(x)|\neq 0, \text{det}\left(D^{2}m_{\partial\Omega}(x)\right)\neq 0$. Here, $m_{\partial\Omega}(x)$ is the restriction of $m(x)$ on $\partial\Omega$.

 \vskip5pt
\item {\rm (iii)} A point $x\in\overline{\Omega}$ is call a \textbf{local maximum}
  of $m$ if $m(x)\geq m(y)$ for every $y$ in a small neighborhood of
  $x$, and there exists some sequence $\{r_{j}\}$ of positive numbers
  such that $\lim_{j\to\infty}r_{j}=0$ and
\[
  m(x)>\max_{y\in \overline{\Omega}\cap\partial B(x,r_{j}) }m(y),\ \
  \forall j\in\mathbb{N}.
\]
\end{definition}

From now on, let us denote
$$\Sigma=\{x\in\overline\Omega:\ x\ \,\mbox{is a local maximum of}\ m\}.$$
Chen and Lou established the following important result;
see \cite[Theorems 1.1]{CL2008}.

\begin{thm}\label{theorem-N} Assume that all critical points of $m$
  are non-degenerate. Let $\lambda_1(\alpha)$ be the principal eigenvalue of \eqref{eq:N}. Then
  \[
    \lim_{\alpha\to\infty}\lambda_1(\alpha)=\min_{x\in\Sigma}V(x).
    \]
\end{thm}

We note that since all critical points of $m$ are non-degenerate, the set $\Sigma$ consists of finitely many isolated points. We further point out that in the setting of Theorem \ref{theorem-N}, as $m\in C^2(\overline\Omega)$, \eqref{div-f} becomes
$\Delta m=0$ on $\overline\Omega$, and the additional assumption that $v\cdot n=0$ in $L_{loc}^1(\partial\Omega)$ imposed in Theorem \ref{theorem-ND}(i) becomes $\frac{\partial m}{\partial n}=0$ on $\partial\Omega$. Thus, $m$ must be a constant and the advection term disappears in \eqref{eq:N}. This shows that Theorem \ref{theorem-ND}(i) does not apply to problem \eqref{eq:N} once $m\in C^2(\overline\Omega)$. However, Theorem \ref{theorem-ND}(i) and (ii) can allow the vector field $v$ to be degenerate so that it has a first integral; for instance, if $v$ vanishes on an open subset $\Omega_0\subset\subset\Omega$, clearly any nonzero function $u\in H_0^1(\Omega_0)$, extended by $0$ in $\Omega\backslash{\Omega_0}$, is a first integral of $v$.

It is also worth mentioning that Theorem \ref{theorem-ND} had been applied in \cite{BHN2005} to investigate the asymptotic behavior of the speeds of propagation of pulsating travelling fronts, and \cite{CL2008} applied Theorem \ref{theorem-N} to reveal new coexistence phenomena of a Lotka-Volterra reaction-diffusion model for two competing species, both with respect to the large advection coefficient.

More recently, in one space dimension (i.e., $\Omega$ is a finite open interval), when the advection term
$m$ admits nature kinds of degeneracy, the first-named author and Zhou \cite{PZ2018} studied an eigenvalue problem with a general boundary condition (see problem \eqref{p} in Section 4). A satisfactory understanding of the limiting behavior of the principle eigenvalue and its eigenfunction was obtained in \cite{PZ2018}. As far as time-periodic parabolic eigenvalue problems are concerned, in the past decades there have been several research works devoted to the study of the asymptotic behavior of the principal eigenvalue, especially with respect to small or large frequency, diffusion coefficient or advection coefficient in spatially or temporally heterogeneous and even degenerate environments. Some of these works include \cite{BHN2023,CKRZ2008,DT2016,Hess1991,HSV2001,HMP2001,LLou,LLPZ2019a,
LLPZ2019b,LLPZ2019c,PP2016,PZZ2019,PZ2015}.

\section{Statement of our main results} In this section, we will present the main results obtained in this paper.
Before going further, let us define the following three subsets of the local maxima set $\Sigma$ of $m$:
\[
  \Sigma_{1}=\Sigma\cap\Omega,\ \ \Sigma_{2}=\{x\in\Sigma\cap\partial\Omega:\ \ \beta(x)=0\},\ \ \Sigma_{3}=\{x\in\Sigma\cap\partial\Omega:\ \ \beta(x)<0\}.
\]
For the eigenvalue problem \eqref{eq:1.1}, we need the following assumption.

\vskip4pt
\noindent\textbf{(A1)}:\ \ Assume that for every $x_{0}\in\Sigma_{2}$, there exists a neighborhood $U$ of
$x_{0}$ such that $\beta(x)=0$ on $\partial\Omega\cap U$.

\vskip4pt
Concerning the eigenvalue problem \eqref{eq:1.1} with nonnegative $\beta$, our result reads as follows.

\begin{theorem}\label{nonnegative} Assume that {\rm(A1)} holds, $\beta\geq 0$ and all critical points of $m$ are non-degenerate. Let $\lambda_1(\alpha)$ be the principal eigenvalue of \eqref{eq:1.1}. Then the following assertions hold.

\item {\rm(i)} If $\Sigma_{1}\cup\Sigma_{2}\neq\emptyset$, we have
\[
\lim_{\alpha\to\infty}\lambda_1(\alpha)=\min_{x\in\Sigma_{1}\cup\Sigma_{2}}V(x).
\]

\item {\rm(ii)} If $\Sigma_{1}\cup\Sigma_{2}=\emptyset$, we have
\[
\lim_{\alpha\to\infty}\lambda_1(\alpha)=\infty.
\]
\end{theorem}

\begin{remark}\label{re-00} {\it Regarding Theorem \ref{nonnegative}, we would like to make some comments as follows.
 \vskip4pt
\item[\rm (i)] Notice that if $\beta=0$ on $\partial\Omega$, then {\rm(A1)} holds automatically, and Theorem \ref{nonnegative} reduces to Theorem \ref{theorem-N} obtained by Chen and Lou.
 \vskip4pt
\item[\rm (ii)] In \cite[Page 478]{BHN2005}, Berestycki, Hamel and Nadirashvili proposed several interesting open questions, and one of them was to ask whether one can find a necessary and sufficient condition for the boundedness of the principal eigenvalue $\lambda_1(\alpha)$ with respect to large $\alpha$ for similar elliptic operators with Robin boundary condition. Our Theorem \ref{nonnegative} above and Theorem \ref{general} below make some progress towards such a question.

}

\end{remark}

For the Dirichlet eigenvalue problem \eqref{eq:1.1-d}, we have the following result.

\begin{theorem}\label{Dirichlet} Assume that  all critical points of $m$ are
  non-degenerate. Let $\lambda_1(\alpha)$ be the principal eigenvalue of \eqref{eq:1.1-d}. Then the following assertions hold.

\item {\rm(i)} If $\Sigma_{1}\neq\emptyset$, we have
\[
\lim_{\alpha\to\infty}{\lambda}(\alpha)=\min_{x\in\Sigma_{1}}V(x).
\]

\item {\rm(ii)} If $\Sigma_{1}=\emptyset$, we have
\[
\lim_{\alpha\to\infty}{\lambda}(\alpha)=\infty.
\]
\end{theorem}

\begin{remark}\label{re-01}{\it Regarding Theorem \ref{Dirichlet}, we want to make some comments as follows.
 \vskip4pt
\item[\rm (i)] Under the assumption that all critical points of $m$ are
  non-degenerate, Theorem \ref{Dirichlet} shows that $\lim_{\alpha\to\infty}{\lambda}(\alpha)=\infty$ if and only if $m$ has no local maximum in the interior of $\Omega$; if $m$ has at least one local maximum in the interior of $\Omega$, then $\lim_{\alpha\to\infty}{\lambda}(\alpha)$ will be determined by Theorem \ref{Dirichlet}(i). Indeed, such assertions remain true as long as $\beta>0$ on $\partial\Omega$ by Theorem \ref{nonnegative}.
 \vskip4pt
\item[\rm (ii)] In dimension $N=2$, if $v$ is a $C^1(\overline\Omega)$ satisfying \eqref{div-f} and $v\cdot n=0$ on $\partial\Omega$, then it is not hard to see that $v$ has a first integral, and therefore $\lim_{\alpha\to\infty}{\lambda}(\alpha)$ is finite and is determined by Theorem \ref{theorem-ND}(ii).
    Nevertheless, if $v\in C^1(\overline\Omega)$ is a gradient field, we take $v=\nabla m$ in \eqref{eq:1.1-d}
  so that \eqref{div-f} becomes $\Delta m=0$ on $\overline\Omega$, and the above-mentioned condition $v\cdot n=0$ on $\partial\Omega$ becomes $\frac{\partial m}{\partial n}=0$ on $\partial\Omega$. In such case, $m$ must be a constant so that $v=0$.

  In dimension $N=1$, that is, $\Omega$ is an open interval, say $\Omega=(0,1)$, then $\Delta m=0$ on $\overline\Omega$ implies $m(x)=ax+b,\, \forall x\in[0,1]$ for some constants $a,\,b$, and so $v(x)=a$ has a first integral in $H_0^1(\Omega)$ if and only if $a=0$. It turns out that this result is a very special case of problem \eqref{p} to be treated in Section 4 below. As mentioned before, for any dimension $N\geq1$, Theorem \ref{theorem-ND}(i) and (ii) can allow $v$ to be degenerate so that it has a first integral.

  The discussion above shows that Theorem \ref{theorem-ND}(ii) does not cover Theorem \ref{Dirichlet}, and vice versa.
  }
\end{remark}

If we allow $\beta$ to change its sign or be negative, generally one
can not expect $\lambda_1(\alpha)$ to be bounded from below with respect to all large $\alpha$.
In order to deal with such a case,  let us denote
$$
\partial\Omega^{-}=\{x\in\partial\Omega:\ \ \beta(x)<0\},\ \ \
\Omega_{\delta}^{-}=\{x\in\Omega:\ \ \text{dist}(x,\partial\Omega^{-})<\delta\}.
$$
We further formulate the following assumption.

\vskip4pt
\noindent\textbf{(A2)}:\ \ There exist a smooth vector field $\nu(x)$
 and a small constant $\delta>0$ such that $\nu(x)\cdot n(x)\geq \delta$
 and $\nabla m(x)\cdot\nu(x)\leq 0$ for all $x\in\Omega_{\delta}^{-}$.
\vskip4pt

For the general boundary function $\beta$, our result can be stated as follows.

\begin{theorem}\label{general}
  Assume that {\rm(A1)} holds and  all critical points of $m$ are
  non-degenerate. Let $\lambda_1(\alpha)$ be the principal eigenvalue of \eqref{eq:1.1}. Then the following assertions hold.

\item {\rm(i)} If $\Sigma_{3}\neq \emptyset$, we have
  \[
  \lim_{\alpha\to\infty} \lambda_1(\alpha)=-\infty.
\]

\item {\rm(ii)} If $\Sigma_{3}=\emptyset$, $\Sigma_{1}\cup\Sigma_{2}\neq
\emptyset$ and {\rm(A2)} holds, we have
\[
  \lim_{\alpha\to\infty} \lambda_1(\alpha)=\min_{\Sigma_{1}\cup\Sigma_{2}}V(x).
\]

\item {\rm(iii)} If $\Sigma_{1}\cup\Sigma_{2}\cup\Sigma_{3}=\emptyset$ and {\rm(A2)} holds, we have
\[
  \lim_{\alpha\to\infty} \lambda_1(\alpha)=\infty.
\]
\end{theorem}

\begin{remark}\label{re0}{\it Notice that $\lambda_1(\alpha)$ must be bounded from below as $\alpha\to\infty$ provided {\rm(A2)} holds; see Lemma \ref{lowerbound} in Section 3.
Moreover, given $x_0\in\Sigma_{3}$, we have $\nabla m(x_0)\cdot \nu(x_0)\geq 0$ for any vector field $\nu(x_0)$ satisfying $\nu(x_0)\cdot n(x_0)>0$, which, together with Theorem \ref{general}(i), implies that {\rm(A2)} is almost a necessary condition to guarantee the boundedness of $\lambda_1(\alpha)$  from below as $\alpha\to\infty$. One may refer to Remark \ref{r-result} in Section 4 for further discussions in dimension $N=1$. On the other hand, $\lambda_1(\alpha)$ must be bounded from above as $\alpha\to\infty$ as long as $\Sigma_{1}\cup\Sigma_{2}\neq \emptyset$; see Lemma \ref{upperbound} below.}
\end{remark}

Let $d(x)=\mathrm{dist}(x,\partial\Omega)$. Then there exists a small constant $d_0>0$ such that $d(x)$ is smooth in $\{x\in\Omega:\ d(x)<d_0\}$. We now extend the definition of $n(x)$ by letting $n(x)=-\nabla d(x)$ for all $x\in \Omega$ satisfying $d(x)\leq d_0$. Thus, for $x\in\partial\Omega$, $n(x)$ is the unit outward normal to $\partial\Omega$.

In all the above stated theorems, the non-degeneracy assumption on all possible critical points of $m$ in $\Omega$ can be relaxed by the following weaker one:

\vskip4pt
\noindent\textbf{(A3)}:\ \ Assume that $m$ has only finitely many critical points. Let $x_{0}\in\overline{\Omega}\setminus\Sigma$ be a critical point of $m$. Assume that the following hold.

\vskip4pt
\noindent\textbf{(i)} If $x_0\in\Omega$, then one of the following is satisfied.

\textbf{(a).} There exist an $N-1$ dimensional  $C^{1}$-surface $\Gamma$ with $x_{0}\in\Gamma$, and positive constants $r_{0}, \delta$ and a smooth unit vector field $\xi(x)$ such that
$\xi(x)\cdot \nu(x)\geq \delta$, and for all $x\in\Gamma\cap B(x_{0},r_{0})$,
  \[ \xi(x)\cdot\nabla m(x)=0\  \text{on}\ \Gamma\cap B(x_{0},r_{0}),\ \
    \xi(x)\cdot\nabla m(x)\geq 0\ \text{ in } D_1,\ \
    \xi(x)\cdot\nabla m(x)\leq 0\ \text{ in }D_2,
  \]
where $\nu(x)$ is the unit normal vector field of $\Gamma$, $D_1$ and $D_2$ are the two connected components of $B(x_0,r_0)\setminus\Gamma$ such that $\nu(x)$ is outward to $D_2$.

\vskip4pt
\textbf{(b).} There exist a smooth unit vector field $\xi(x)$ and a constant $r_{0}>0$ such that $\xi(x)\cdot\nabla m(x)\geq 0$ for all $x\in B(x_{0},r_{0})\cap\Omega$.
\\

\noindent\textbf{(ii)} If $x_0\in\partial\Omega$, then either $x_0$ is non-degenerate or one of the following is satisfied.

\textbf{(a).} There exist a smooth unit vector field $\xi(x)$ and positive constants $r_0$ and $\delta$ such that
\[
\xi(x)\cdot n(x)\geq\delta\ \, \text{ on }\partial\Omega\cap B(x_0,r_0),\ \
\nabla m(x)\cdot \xi(x)\leq 0\ \, \text{ in }\Omega\cap B(x_0,r_0).
\]

\textbf{(b).} There exist a smooth unit vector field $\xi(x)$ and positive constants $r_0$ and $\delta$ such that
\[
\xi(x)\cdot n(x)\leq 0\ \, \text{ on }\partial\Omega\cap B(x_0,r_0),\ \
\nabla m(x)\cdot \xi(x)\leq 0\ \, \text{ in }\Omega\cap B(x_0,r_0).
\]

\textbf{(c).} There exist an $N-1$ dimenstional $C^1$-surface
$\Gamma$, a smooth unit vector field $\xi(x)$ and positive constants
$r_0$ and $\delta$ such that  $x_0\in\Gamma$, $|n(x_0)\cdot \nu(x_0)|\ne
1$, $\xi(x)\cdot \nu(x)>\delta$ on $\Gamma\cap B(x_0,r_0)$, $
\xi(x)\cdot \nabla m(x)=0$ for all $x\in\Gamma\cap B(x_0,r_0)$, and
\[
\xi(x)\cdot n(x)\leq0\ \, \text{ on }\partial\Omega\cap \overline{D}_1,
\ \
\nabla m(x)\cdot \nu(x)\geq 0\ \, \text{ in } D_1,
\]
 and
\[
\xi(x)\cdot n(x)\geq 0 \ \, \text{ on }\partial\Omega\cap \overline{D}_2,\ \
\nabla m(x)\cdot \nu(x)\leq0\ \, \text{ in } D_2,
\]
where $\nu(x)$ is the unit normal vector field of $\Gamma$, $D_1$ and $D_2$ are the two connnected components of $\left(\Omega\cap B(x_0,r_0)\right)\setminus\Gamma$ such that $n(x)$ is outward to $D_2$.

\begin{remark}\label{re1}{\it  We would like to provide some concrete examples so that {\rm(A3)} holds.

\item {\rm(i)} Suppose that $x_{0}\in\Omega\setminus\Sigma$ is a non-degenerate critical
  point of $m$, then it is easy to check that $m$
  satisfies (i)-(a) of {\rm(A3)} at $x_{0}$.

 \vskip4pt
\item {\rm(ii)} Suppose that $x_{0}\in\Omega$ is a critical point of $m$ and
  \[
m(x)=m(x_{0})+c(e\cdot(x-x_{0}))^{3}+o(|x|^{3})
\]
in a small neighborhood of $x_{0}$ with some constant $c\neq 0$ and
$e\in S^{N-1}$, where $S^{N-1}$ is the unit sphere of $ \mathbb{R}^{N}$.
Then $m$ satisfies (i)-(b) of {\rm(A3)} at $x_{0}$.

\vskip4pt
\item {\rm(iii)} Suppose that $x_{0}\in\partial\Omega$ and
  \[
m(x)=m(x_{0})+c_{1}n(x_{0})\cdot (x-x_{0})+c_{2}(\tilde\nu\cdot(x-x_{0}))^{3}+o(|x|^{3})
\]
in a small neighborhood of $x_{0}$ with some constants $c_{1}>0,
c_{2}\neq 0$, where $\tilde\nu$ is a tangent vector of $\Gamma$ at $x_{0}$.
Then $m$ satisfies (ii)-(a) of {\rm(A3)} at $x_{0}$.

\vskip4pt
\item {\rm(iv)} Suppose that $x_{0}\in\partial\Omega$  and
  \[
m(x)=m(x_{0})+c(n(x_{0})\cdot(x-x_{0}))^{3}+o(|x|^{3})
\]
in a small neighborhood of $x_{0}$ with some constant $c<0$.
Then $m$ satisfies (ii)-(b) of {\rm(A3)} at $x_{0}$.

\vskip4pt
\item {\rm(v)} Suppose that $x_{0}\in\partial\Omega$, $\Omega$ is concave
  in a neighborhood of $x_{0}$, and
  \[
m(x)=m(x_{0})+c_{1}n(x_{0})\cdot (x-x_{0})+c_{2}(\tilde\nu\cdot(x-x_{0}))^{4}+o(|x|^{4})
\]
in a small neighborhood of $x_{0}$ with some constants $c_{1}\geq 0,
c_{2}>0$, where $\tilde\nu$ is a tangent vector of $\Gamma$ at $x_{0}$.
Then $m$ satisfies (ii)-(c) of {\rm(A3)} at $x_{0}$.

}
\end{remark}

Theorem \ref{nonnegative} and Theorem
\ref{general} can be improved as follows.

\begin{theorem}\label{degenerate}
  Assume that {\rm(A1)} and {\rm(A3)} hold. Let $\lambda_1(\alpha)$ be the principal eigenvalue of \eqref{eq:1.1}. Then the following assertions hod.

\item {\rm(i)} If $\Sigma_{3}\neq \emptyset$, we have
  \[
  \lim_{\alpha\to\infty} \lambda_1(\alpha)=-\infty.
\]

\item {\rm(ii)} If $\Sigma_{3}=\emptyset$, $\Sigma_{1}\cup\Sigma_{2}\neq
\emptyset$ and {\rm(A2)} holds, we have
\[
  \lim_{\alpha\to\infty} \lambda_1(\alpha)=\min_{\Sigma_{1}\cup\Sigma_{2}}V(x).
\]

\item {\rm(iii)} If $\Sigma_{1}\cup\Sigma_{2}\cup\Sigma_{3}=\emptyset$ and {\rm(A2)} holds, we have
\[
  \lim_{\alpha\to\infty} \lambda_1(\alpha)=\infty.
\]
\end{theorem}

We want to make the following remarks.
\begin{remark}\label{re-11}{\it

\item{\rm(i)} Notice that in the Neumann boundary problem \eqref{eq:N}, our assumptions {\rm(A1)} and {\rm(A2)} hold automatically. In addition, we do not require the non-degeneracy of local maxima of $m$. Thus, under the weaker condition {\rm(A3)}, the assertion (ii) of Theorem \ref{degenerate} improves Theorem \ref{theorem-N} obtained by Chen and Lou in \cite{CL2008}.

 \vskip4pt
\item{\rm(ii)} Both Theorem \ref{general}(i) and Theorem \ref{degenerate}(i) remain true only if the set $\Sigma_{3}$ contains one local maximum in the sense of Definition \ref{def-1}(iii) and all the critical points of $m$ allow to be degenerate; see Remark \ref{re2} below.

   \vskip4pt
\item{\rm(iii)} Theorem \ref{Dirichlet} remains valid provided the non-degeneracy assumption of all critical points of $m$ is replaced by {\rm(A3)}.
    }

\end{remark}

\begin{remark}\label{re-p}{\it For any dimension $N\ge1$, let $L_i\,(1\leq i\leq N)$ be given positive numbers. Assume that $m\in C^2(\mathbb{R}^N)$ and $V\in C(\mathbb{R}^N)$ are $L$-periodic with respect to the spatial variable $x$ in the sense of \cite{BH2002,BHN2005}; that is, $w(x_1+L_1,x_2+L_2,\cdots,x_n+L_N)=w(x_1,x_2,\cdots,x_N)$ for all $x=(x_1,x_2,\cdots,x_N)\in\mathbb{R}^N$ and $w\in\{m,\, V\}$. Consider the following spatial-periodic eigenvalue problem:
\begin{equation}\label{eq:periodic}
\left\{
\begin{aligned}
  &-\Delta \varphi -2\alpha\nabla m(x)\cdot \nabla\varphi+V(x)\varphi=\lambda\varphi &&\hbox{ in }\mathbb{R}^N,\\
  &\varphi \ \mbox{is $L$-periodic in}\ x.
\end{aligned}
\right.
\end{equation}

Let $\Sigma$ be the local maxima set of $m$ in the sense of Definition \ref{def-1}(iii) and denote by $\lambda_1(\alpha)$ the unique principal eigenvalue of \eqref{eq:periodic}. Assume that $m$ has only finitely many critical points on $\{x=(x_1,x_2,\cdots,x_N)\in\mathbb{R}^N:\ x_i\in(0,2L_i),\ 1\leq i\leq N\}$. For any critical point $x_{0}$ of $m$ with $x_{0}\in\mathbb{R}^N\setminus\Sigma$, we further assume that
either \textbf{(A3)}(i)-(a) or (i)-(b) holds. Then, by the similar analysis to prove Theorem \ref{degenerate}, we can conclude that
\[
  \lim_{\alpha\to\infty} \lambda_1(\alpha)=\min_{\Sigma}V(x).
\]

This result is not covered by \cite{BHN2005} and gives an explicit limit, compared to \cite[Theorem 2.2]{BHN2005}.
}
\end{remark}

The primary results of this paper reveal how boundary conditions significantly impact the asymptotic behavior of the principal eigenvalue in the context of large advection rates.

Regarding the proofs of our main results, though our approach is inspired by \cite{CL2008}, we encounter several significant challenges that require us to develop some new techniques to overcome. One of such challenges is to determine whether the following energy functional
\begin{equation}\label{energy}
\int_{\Omega}\{|\nabla w-\alpha w\nabla m|^{2}\}+\int_{\partial\Omega}\beta w^{2}
\end{equation}
is bounded or not for all large $\alpha$, where $w$ is the principal eigenfunction corresponding to $\lambda_1(\alpha)$ with $\int_\Omega w^2=1$  for all $\alpha\geq0$. When $\beta=0$, it is clear to see that \eqref{energy} is bounded from both above and below. This is one of the key points in \cite{CL2008} to derive Theorem \ref{theorem-N}. Nevertheless, such a property does not hold for a general boundary function $\beta$.

In the case of $\beta\geq0$, obviously \eqref{energy} is bounded from below. Furthermore, as long as $\Sigma_{1}\cup\Sigma_{2}\neq \emptyset$ holds, \eqref{energy} is also bounded from above (see Lemma \ref{upperbound}). In fact, thanks to Lemma \ref{upperbound}, in this case, in order to prove Theorem \ref{nonnegative}(i), we just need to determine $\liminf_{\alpha\to\infty}\lambda_1(\alpha)$. This will be achieved by analyzing the support of the weak limit $\mu$ of $w$, which is a probability measure. The primary task is to show the support of $\mu$ satisfies $\text{supp}(\mu)\subset\Sigma_{1}\cup\Sigma_{2}$, which involves some nontrivial analysis. If $\Sigma_{1}\cup\Sigma_{2}=\emptyset$, we shall show that \eqref{energy} must be unbounded from above and in turn $\lim_{\alpha\to\infty}\lambda_1(\alpha)=\infty$ by resorting to
a contradiction argument so as to lead to $\text{supp}(\mu)=\emptyset$.

The case that $\beta$ may be negative somewhere on $\partial\Omega$ turns out to be more challenging to handle since \eqref{energy} may be unbounded from either below or above. One of the crucial ingredients in our argument is establishing a trace-type inequality (see Lemma \ref{lemma-trace}). This inequality allows us to control the boundary integral term in \eqref{energy} by the first integral term of \eqref{energy}.  As a result, under proper assumptions on the advection function $m$, we can determine whether \eqref{energy} is bounded or unbounded; based on such information, $\text{supp}(\mu)$ can be therefore clarified.

All of the aforementioned results apply to any dimension ($N\geq1$). For the one-dimensional case ($N=1$), we are able to give a complete description of the limiting behavior of the principal eigenvalue of the associated eigenvalue problem (i.e., \eqref{p} in Section 4), by allowing the advection function $m$ to have natural kinds of degeneracy. Our results exhaust all parameter ranges of the boundary conditions and thus complement those in \cite{PZ2018}. See Section 4 for precise details. We want to stress that Lemma \ref{lemma-trace}, which provides a trace inequality, is also crucial in yielding the results of the one-dimensional problem \eqref{p}. Additionally, we note that the principal eigenvalue of the periodic eigenvalue problem \eqref{eq:periodic} with $N=1$ and the degenerate function $m$ was treated in \cite{PZ2018}; see Theorem 2.13 there.

The rest of our paper is organized as follows. In Section 3, we will give the proofs of Theorem \ref{nonnegative}-\ref{degenerate}. In Section 4, we will consider the one-dimensional problem and obtain a complete characterization of the asymptotic behavior of the principal eigenvalue, and some typical examples will also be used to demonstrate the impact of boundary conditions and the degeneracy of $m$.

\section{Proofs of Theorems \ref{nonnegative}-\ref{degenerate}} This
section is devoted to the proofs of our main results: Theorems
\ref{nonnegative}-\ref{degenerate}. Without loss of generality, we
always take $D=1$ from now on. Throughout this paper, we use $S^{k-1}$
to denote the unit sphere of $ \mathbb{R}^{k}$, $\mathcal{H}^{k}$ to
denote the $k$ dimensional Hausdorff measure, and $f^+$ and $f^-$, respectively, to denote the positive part and negative part of a function $f$.

\subsection{Two lemmas}
For later purpose, we need to make the following transformation:
$$
w(x)=e^{\alpha m(x)}\varphi(x)
$$
and also normalize $w$ with $\int_{\Omega}w^2dx=1$. Then, in view of \eqref{eq:1.1} satisfied by $\varphi$, direct computation shows that $w$ solves
\begin{equation}
  \label{eq:3.3}
  \left\{
\begin{aligned}
  &-\Delta w+\left(\alpha^2|\nabla m|^2+\alpha\Delta m+V-\lambda\right)w=0 &&\hbox{ in }\Omega,\\
  &\frac{\partial w}{\partial n}-\alpha w\frac{\partial
    m}{\partial n}+\beta w=0\ && \hbox{ on
  }\partial\Omega,\\
  &\int_{\Omega}w^2=1.
\end{aligned}
\right.
\end{equation}

It is also well known that the principal eigenvalue $\lambda_1(\alpha)$ can be variationally characterized by
\begin{equation}\label{eq:1.2}
  \begin{split}
    \lambda_1(\alpha)&=\inf_{\varphi\in
                     H^{1}(\Omega),\,\varphi\not=0}\frac{\int_{\Omega}e^{2\alpha
                     m}(|\nabla\varphi|^{2}+V\varphi^{2})dx+\int_{\partial\Omega}\beta
                     e^{2\alpha
                     m}\varphi^{2}d\mathcal{H}^{N-1}}{\int_{\Omega}e^{2\alpha
                     m}\varphi^{2}dx}\\
&=\inf_{w\in
                   H^{1}(\Omega),\,\int_{\Omega}w^{2}dx=1}\int_{\Omega}|\nabla
                   w-\alpha w\nabla
                   m|^{2}+Vw^{2}dx+\int_{\partial\Omega}\beta w^{2}d\mathcal{H}^{N-1}.
\end{split}
\end{equation}
One may refer to \cite{CL2008,CL2012,PZZ2019}.

We first establish an estimate for the upper bound of $\lambda_1(\alpha)$ once $\Sigma_{1}\cup\Sigma_{2}\neq\emptyset$.

\begin{lemma}\label{upperbound}
If $\Sigma_{1}\cup\Sigma_{2}\neq\emptyset$, then it holds that
\[
  \limsup_{\alpha\to\infty}\lambda_1(\alpha)\leq\min_{x\in\Sigma_{1}\cup\Sigma_{2}}V(x).
\]
\end{lemma}

\begin{proof} Our analysis is similar to that of \cite[Lemma 2.4]{CL2008}. Fix $x_{0}\in\Sigma_{1}\cup\Sigma_{2}$. Then, by our Definition \ref{def-1}(iii), there exists a sequence
  $\{r_{j}\}$ such that $\lim_{j\to\infty}r_{j}=0$ and
\[
  m(x_{0})>\max_{\overline{\Omega}\cap\partial B(x_{0},r_{j})}m(x),\ \ \ \forall j\in\mathbb{N}.
\]

Now, we first choose a sequence $\{s_{j}\}$ such that
\[
  0<s_{j}<r_{j},\ \ \min_{\overline{\Omega}\cap
    \overline{B(x_{0},s_{j})}}m(x):=m_{j}>\max_{\overline{\Omega}\cap\partial
    B(x_{0},r_{j})}m(x),
\]
and then choose another sequence $\{t_{j}\}$ such that
\[
  s_{j}<t_{j}<r_{j},\ \ \min_{\overline{\Omega}\cap
    \overline{B(x_{0},s_{j})}}m(x):=m_{j}>M_{j}:=\max_{\overline{\Omega}\cap
    \overline{B(x_{0},r _{j})}\setminus B(x_{0},t_{j})}m(x).
\]

We define
\begin{equation*}
  \varphi_{j}(x)=
  \left \{
   \begin{aligned}
    & 1&& \text{ if }x\in \overline{B(x_{0},t_{j})},\\
    & (r_{j}-|x|)/(r_{j}-t_{j}) && \text{ if } x\in
                                     \overline{B(x_{0},r_{j})}\setminus
                                     B(x_{0},t_{j}),\\
     &0&& \text{ if }x\in \mathbb{R}^{N}\setminus B(x_{0},r_{j}).
   \end{aligned}
   \right.
 \end{equation*}

In the case that $x_{0}\in\Sigma_{2}$, according to the assumption (A1), there exists $\delta>0$ such that $\beta(x)=0$ for all $x\in\partial\Omega\cap B(x_{0},\delta)$. Thus, in either case of $x_{0}\in\Sigma_{1}$ or $x_{0}\in\Sigma_{2}$, we may assume that, for sufficiently large $j$,
$\beta(x)\varphi_{j}(x)\equiv 0$ on $\partial\Omega$. Direct calculation
yields from \eqref{eq:1.2} that
 \begin{align*}
     \lambda_1(\alpha)&\leq \frac{\int_{\Omega}e^{2\alpha m}|\nabla
                      \varphi_{j}|^{2}dx}{\int_{\Omega}e^{2\alpha
                      m}\varphi_{j}^{2}dx}+\frac{\int_{\Omega}e^{2\alpha
                      m}V\varphi_{j}^{2}dx}{\int_{\Omega}e^{2\alpha
                      m}\varphi_{j}^{2}dx}+\frac{\int_{\partial\Omega}e^{2\alpha
                      m}\beta\varphi_{j}^{2}d\mathcal{H}^{N-1}}{\int_{\Omega}e^{2\alpha
                      m}\varphi_{j}^{2}dx}\\
     &\leq\frac{e^{2\alpha M_{j}}r_{j}^{N}}{|r_{j}-t_{j}|^{2}s_{j}^{N}e^{2\alpha
       m_{j}}}+\max_{\overline{B(x_{0},r_{j})}}V(x)
 \end{align*}
 for sufficiently large $j$. Here we used the first variational characterization of $\lambda_1(\alpha)$ in \eqref{eq:1.2}. Letting $\alpha\to \infty$ and then $j\to\infty$, we obtain that
 \[
   \limsup_{\alpha\to\infty}\lambda_1(\alpha)\leq V(x_{0}),
 \]
which completes the proof.
\end{proof}

We next consider the case of $\Sigma_{3}\neq\emptyset$ and obtain the following result.

\begin{lemma}\label{unbounded}
  If $\Sigma_{3}\neq\emptyset$, then it holds that
  \[
\lim_{\alpha\to\infty}\lambda_1(\alpha)=-\infty.
    \]
  \end{lemma}
  \begin{proof} Let $x_{0}\in\Sigma_{3}$. Then there exists an $r_{0}>0$ such that
   \begin{equation}\label{add-1}
   \mbox{ $\beta(x)\leq \frac{\beta(x_{0})}{2}:=-\delta$\ \ \ \ for all\
    $x\in\partial\Omega\cap B(x_{0},r_{0})$.}
 \end{equation}
We choose the sequences
    $\{r_{j}\},\{s_{j}\}$ and $\{t_{j}\}$ and the functions $\varphi_{j}$
    the same as in the proof of Lemma \ref{upperbound}. For any given $\epsilon>0$, there exists a
    constant $M$ such that, for $\alpha,j>M$,
    \begin{equation}\label{add-2}
      \frac{\int_{\Omega}e^{2\alpha m}|\nabla
        \varphi_{j}|^{2}dx}{\int_{\Omega}e^{2\alpha
          m}\varphi_{j}^{2}dx}+\frac{\int_{\Omega}e^{2\alpha
          m}V\varphi_{j}^{2}dx}{\int_{\Omega}e^{2\alpha
          m}\varphi_{j}^{2}dx}\leq V(x_{0})+\epsilon.
 \end{equation}
Without loss of generality, we may assume that $x_{0}=0$ and
$n(x_{0})=(0,\cdots,0,-1)$. Let $r=|x|$, $f_{j}(r)=\varphi_{j}(x)$. We
define $\Phi:\partial\Omega\cap B(0,r_{0})\to B'(0',r_{0})$,
\[
  \Phi(x)=\frac{|x|}{|x'|}x',
\]
where $x=(x',x_{N})$,
$B'(0,r_{0})=\{x'\in\mathbb{R}^{N-1}:\ \ |x'|<r_{0}\}=B(0,r_{0})\cap\mathbb{R}^{N-1}$. Since
$\partial\Omega$ is smooth, $\Phi$ is a diffeomorphism from
$\partial\Omega\cap B(0,r_{0})$ to $B'(0',r_{0})$ for sufficiently small
$r_{0}$. Moreover, there exists a constant $C$ such that
    \begin{equation}\label{add-3}
 C^{-1} \int_{B'(0,r_{0})}|u(\Phi^{-1}(x))|dx'\leq\int_{\partial\Omega\cap
    B(0,r_{0})}|u(x)|d\mathcal{H}^{N-1}\leq C \int_{B'(0,r_{0})}|u(\Phi^{-1}(x))|dx'
 \end{equation}
for all $u\in L^{1}(\partial\Omega\cap B(0,r_{0}))$.

Notice that $\varphi_{j}(\Phi^{-1}(x))=\varphi_{j}(x)=f_{j}(r)$. We thus deduce from \eqref{add-3} that
\begin{align*}\nonumber
  \lim_{j\to\infty}\frac{\int_{\partial\Omega}\varphi_{j}^{2}d\mathcal{H}^{N-1}}{\int_{\Omega}\varphi_{j}^{2}dx}
  &=\lim_{j\to\infty}\frac{\int_{\partial\Omega\cap B({0},r_{0})}\varphi_{j}^{2}d\mathcal{H}^{N-1}}
  {\int_{\mathbb{R}^{N}}\varphi_{j}^{2}dx}\nonumber\\
  &
  \geq\lim_{j\to\infty}\frac{1}{C}\frac{\int_{\mathbb{R}^{N-1}}\varphi_{j}^{2}dx}
  {\int_{\mathbb{R}^{N}}\varphi_{j}^{2}dx}\nonumber\\
&=\lim_{j\to\infty}\frac{1}{C}\frac{\int_{S^{N-2}}\int_{0}^{r_{j}}f_{j}(r)r^{N-2}drd\mathcal{H}^{N-2}}
{\int_{S^{N-1}}\int_{0}^{r_{j}}f_{j}(r)r^{N-1}drd\mathcal{H}^{N-1}}\nonumber\\
&\geq
  \lim_{j\to\infty}\frac{1}{C}\frac{\mathcal{H}^{N-2}(S^{N-2})}{\mathcal{H}^{N-1}(S^{N-1})}\frac{1}{r_{j}}\nonumber\\
&=\infty.
\end{align*}
By fixing $\alpha>M$, due to \eqref{add-1}, one can easily see that
     \begin{equation}\label{add-4}
        \lim_{j\to\infty}\frac{\int_{\partial\Omega}\beta e^{2\alpha
            m}\varphi_{j}^{2}d\mathcal{H}^{N-1}}{\int_{\Omega}e^{2\alpha
            m}\varphi_{j}^{2}dx}\leq
        -\delta\frac{\min_{\overline{B({0},r_{j})}}e^{2\alpha m}}{e^{2\alpha
            \max_{x\in\overline\Omega}|m({x})|}}\lim_{j\to\infty}\frac{\int_{\partial\Omega}\varphi_{j}^{2}d\mathcal{H}^{N-1}}{\int_{\Omega}\varphi_{j}^{2}dx}
        =-\infty.
    \end{equation}
      Therefore, by \eqref{add-2} and \eqref{add-4}, we have
      \begin{align*}
     \lambda_1(\alpha)&\leq \frac{\int_{\Omega}e^{2\alpha m}|\nabla
                      \varphi_{j}|^{2}dx}{\int_{\Omega}e^{2\alpha
                      m}\varphi_{j}^{2}dx}+\frac{\int_{\Omega}e^{2\alpha
                      m}V\varphi_{j}^{2}dx}{\int_{\Omega}e^{2\alpha
                      m}\varphi_{j}^{2}dx}+\frac{\int_{\partial\Omega}e^{2\alpha
                      m}\beta\varphi_{j}^{2}d\mathcal{H}^{N-1}}{\int_{\Omega}e^{2\alpha
                      m}\varphi_{j}^{2}dx}\\
     &\leq V({0})+\epsilon-\delta\frac{\min_{\overline{B({0},r_{j})}}e^{2\alpha m}}{e^{2\alpha
            \max_{x\in\overline\Omega}|m({x})|}}\frac{\int_{\partial\Omega}\varphi_{j}^{2}d\mathcal{H}^{N-1}}{\int_{\Omega}\varphi_{j}^{2}dx}\\
            &\to
        -\infty,\ \ \ \mbox{as}\ j\to\infty.
 \end{align*}
This completes the proof.
\end{proof}

\begin{remark}\label{re2}{\it The proof of Lemma \ref{unbounded} does not use the degeneracy of critical points of $m$ in $\Omega$; that is, Lemma \ref{unbounded} holds as long as  $\Sigma_{3}$ contains an isolated local maximum of $m$.}
\end{remark}

\subsection{Proof of Theorem \ref{nonnegative}}
In this subsection, we assume that $\beta\geq 0$ and aim to prove Theorem \ref{nonnegative}.

Before going further, let us introduce some notations to be used frequently later. Let $\varphi_{\alpha}$ be the eigenfunctions corresponding to the principal eigenvalue $\lambda_1(\alpha)$, and $w_{\alpha}=e^{\alpha
    m}\varphi_{\alpha}$. Without loss of generality, we may assume that
 $\int_{\Omega}w_{\alpha}^{2}dx=1$ for all $\alpha$.

For each $\alpha\in\mathbb{R}$, we extend $w_{\alpha}$ to all
 $\mathbb{R}^{N}$ by setting it to be zero on
 $\mathbb{R}^{N}\setminus\Omega$. We define a measure $\mu_{\alpha}$
 on $\mathbb{R}^{N}$ by
 \[
   \mu_{\alpha}(A)=\int_{A}w_{\alpha}^{2}dx.
 \]
 Then $\mu_{\alpha}$ is a Radon measure in $\mathbb{R}^{N}$  with
 $\mathrm{supp}(\mu_{\alpha})\subset \overline{\Omega}$ and
   $\mu_{\alpha}(\overline{\Omega})=1$, where $\mathrm{supp}(\mu_{\alpha})$
   denotes the support of $\mu_{\alpha}$. By the weak compactness of Radon
 measures, there exist a sequence $\{\alpha_{k}\}_{k\in\mathbb{N}}$
 and a Radon measure $\mu$ such that $\alpha_{k}\to \infty$ as
 $k\to\infty$ and $\mu_{\alpha_{j}}$ weakly converges to $\mu$ in the sense of
  \begin{equation}\label{weakcon}
  \lim_{k\to\infty}\int_\Omega w_{\alpha_k}^{2}(x)\zeta(x)dx=\int_{\overline\Omega}\zeta(x)d\mu,\  \ \ \forall \zeta\in C(\overline\Omega),
   \end{equation}
fulfilling $\mathrm{supp}(\mu)\subset\overline\Omega$ and $\mu(\overline\Omega)=1$. One may refer to \cite{EG2015}.

To prove Theorems \ref{nonnegative}-\ref{degenerate}, as it will be seen later,
the key ingredient of our mathematical analysis is to gain a precise understanding of $\mathrm{supp}(\mu)$.
To the end, as in \cite{CL2008}, we denote the set of non-critical interior
points of $m$ by
$$
\Omega_{1}=\{x\in\Omega:\ \ |\nabla m(x)|\neq 0\},
$$
and the set of non-degenerate interior
critical points which are not local maxima by
$$
\Omega_{2}=\{x\in\Omega:\ \ |\nabla m(x)|=0,\ \exists e\in S^{N-1}\ \text{
     such that}\ (e\cdot\nabla)^{2}m(x)>0\},
$$
and the set of non-critical boundary
points by
$$
\Omega_{3}=\left\{x\in\partial\Omega:\ \ |\nabla m(x)|>\left|\frac{\partial m}{\partial n}(x)\right|\right\},
$$
and divide all boundary critical points which are not local maxima into
the following three categories:
 \begin{equation*}
 \begin{aligned}
   &\Omega_{4}=\left\{x\in\partial\Omega:\ \ |\nabla m(x)|=-\frac{\partial m}{\partial n}(x)>0\right\},\\
   &\Omega_{5}=\left\{x\in\partial\Omega:\ \ |\nabla m(x)|=\frac{\partial m}{\partial n}(x)>0,\ \exists e\in
     S^{N-1}\ \text{such that} \ e\perp n,\ (e\cdot\nabla)^{2}m(x)>0\right\},\\
   &\Omega_{6}=\left\{x\in\partial\Omega:\ \ |\nabla m(x)|=0,\ \exists e\in
     S^{N-1}\ \text{such that}\ (e\cdot\nabla)^{2}m(x)>0\right\}.
 \end{aligned}
\end{equation*}

With the above preparation, we are now ready to present the proof of Theorem \ref{nonnegative}.

\begin{proof}[Proof of Theorem \ref{nonnegative}] We first prove (i). The upper bound estimate
  \begin{equation}\label{ineq1}
\limsup_{\alpha\to\infty}\lambda_1(\alpha)\leq \min_{\Sigma_{1}\cup\Sigma_{2}}V(x)
\end{equation}
follows directly from Lemma \ref{upperbound}.

In what follows, we shall prove
\begin{equation}\label{ineq1-a}
\liminf_{\alpha\to\infty}\lambda_1(\alpha)\geq \min_{\Sigma_{1}\cup\Sigma_{2}}V(x).
\end{equation}
If we can show $\text{supp}(\mu)\subset\Sigma_{1}\cup\Sigma_{2}$, it then follows that
\begin{equation}\label{ineq2}
  \begin{split}
    \liminf_{\alpha\to\infty}\lambda_1(\alpha)&=\liminf_{\alpha\to\infty}\frac{\int_{\Omega}e^{2\alpha
        m}|\nabla\varphi_{\alpha}|^{2}+e^{2\alpha
        m}V\varphi_{\alpha}^{2}dx+\int_{\partial\Omega}\beta e^{2\alpha
        m}\varphi_{\alpha}^{2}d\mathcal{H}^{N-1}}{\int_{\Omega}e^{2\alpha
        m}\varphi_{\alpha}^{2}dx}\\
    &\geq\liminf_{\alpha\to\infty}\frac{\int_{\Omega}e^{2\alpha
        m}V\varphi_{\alpha}^{2}dx}{\int_{\Omega}e^{2\alpha
        m}\varphi_{\alpha}^{2}dx}\\
    &=\liminf_{\alpha\to\infty}\frac{\int_{\Omega}Vw_{\alpha}^{2}dx}{\int_{\Omega}w_{\alpha}^{2}dx}\\
    &=\frac{\int_{\overline{\Omega}}Vd\mu}{\mu(\overline{\Omega})}\\
    &\geq \min_{\Sigma_{1}\cup\Sigma_{2}}V(x),
\end{split}
\end{equation}
which verifies \eqref{ineq1-a}. Hence, a combination of \eqref{ineq1} and \eqref{ineq2} gives
\[
\lim_{\alpha\to\infty}\lambda_1(\alpha)=\min_{x\in\Sigma_{1}\cup\Sigma_{2}}V(x),
\]
as wanted.

Let $\Omega_{i}\ (1\leq i\leq 6)$ be defined as before, and further set
$$
 \Sigma_{4}=\{x\in\Sigma\cap\partial\Omega:\ \ \beta(x)>0\}.
 $$
In order to show $\text{supp}(\mu)\subset\Sigma_{1}\cup\Sigma_{2}$, it is not hard to see that we just need to prove $\mu(\Sigma_{4})=0$ and $\mu(\Omega_{i})=0\ (1\leq i\leq 6)$.

By means of \eqref{ineq1} and $\beta\geq0$, we observe that
\begin{equation}\label{bound-abc}
\int_{\Omega}|\nabla w_{\alpha}-\alpha w_{\alpha}\nabla
m|^{2}dx\leq C
 \end{equation}
for some positive constant $C$, which does not depend on $\alpha\geq0$.
Thus, making use of \eqref{bound-abc}, one can follow the same arguments as in \cite[Lemmas 3.1-3.4 and Lemma 4.1]{CL2008} to claim that $\mu(\Omega_{i})=0$ for each $1\leq i\leq 6$.
Thus, it remains to verify $\mu(\Sigma_{4})=0$.

In the sequel, we are going to prove that
\begin{equation}\label{bound-ab4}
\mu(\Sigma_4)=0.
 \end{equation}
By our assumption of the non-degeneracy of critical points, $\Sigma_4$ is a discrete set. It is sufficient to show that $\mu(\{x_{0}\})=0$ for each
$x_{0}\in\Sigma_{4}$.

By \eqref{ineq1} and $\int_{\Omega}w_{\alpha}^{2}dx=1$, it is clear that
 \begin{equation}\label{bound-a}
\int_{\Omega}|\nabla w_{\alpha}-\alpha w_{\alpha}\nabla
m|^{2}dx+\int_{\partial\Omega}\beta w_{\alpha}^{2}d\mathcal{H}^{N-1}\leq C_{*}
 \end{equation}
for some positive constant $C_{*}$, independent of $\alpha\geq0$.

We may assume that $x_{0}=0$, and shall prove $\mu(\{0\})=0$ by a contradiction argument. Supposing that $\mu(\{0\})>0$, we will show that
either
\[
\limsup_{\alpha\to\infty}\int_{\partial\Omega}\beta w_{\alpha}^{2}d\mathcal{H}^{N-1}=\infty,
\]
or
\[
\limsup_{\alpha\to\infty}\int_{\Omega}|\nabla w_{\alpha}-\alpha
w_{\alpha}\nabla m|^{2}dx=\infty
\]
holds, which leads to a contradiction against the fact \eqref{bound-a}.

We assume that
 \begin{equation}\label{bound-b}
\limsup_{\alpha\to\infty}\int_{\partial\Omega}\beta w_{\alpha}^{2}d\mathcal{H}^{N-1}<\infty,
 \end{equation}
otherwise there is nothing to do.

In the following, we have to distinguish two different cases.

\vskip4pt \textbf{Case 1.} $\partial_{n}m(0)=|\nabla m(0)|>0$ and
$(e\cdot\nabla)^{2}m<0$ for all $e\perp n(0)$, $e\in S^{N-1}$. In this case, by a rotation, we may assume that $n(0)=(0,\cdots,0,-1)$. Near the point $0$, the boundary of
$\Omega$ can be expressed by
\[
  \{x=(x',x_{N}=\psi(x'))\},\ \ \ \text{ where }\ \,
  x'=(x_{1},\cdots,x_{N-1})
  \]
with $\psi(x')\in C^{2}(\mathbb{R}^{N-1})$ and
$\nabla_{x'}\psi(0')=0$. Locally the boundary can be flattened by a
simple diffeomorphisim
\[
  x=X(z):=(z',\psi(z')+z_{N}),\ \ z=Z(x)=(x',x_{N}-\psi(x')).
\]
It is easy to see that
\[
\det\left(\frac{\partial X(z)}{\partial z}\right)=1,\ \ \
\det\left(\frac{\partial Z(x)}{\partial x}\right)=1.
  \]

Let $\tilde{w}_{\alpha}(z)=w_{\alpha}(X(z))$, $\tilde{\varphi}_{\alpha}(z)=\varphi_{\alpha}(X(z))$ and $\tilde{m}(z)=m(X(z))$. We choose a small constant
$r_{0}>0$ such that, for $z\in B'(0',r_{0})$, $|\nabla \psi(z')|<1$ and
\[
  \beta(X(z))\geq \frac{\beta(0)}{2}:=\delta>0.
\]
Since
\[
\int_{X(B'(0',r_{0}))}w_{\alpha}^{2}d\mathcal{H}^{N-1}=\int_{B'(0',r_{0})}\tilde
w_{\alpha}^{2}\left(1+|\nabla_{x'}\psi|^{2}\right)^{1/2}dx',
\]
we have
\[
\int_{B'(0',r_{0})}\tilde
w_{\alpha}^{2}dx'\leq \int_{X(B'(0',r_{0}))}w_{\alpha}^{2}d\mathcal{H}^{N-1}\leq
\frac{1}{\delta}\int_{X(B'(0',r_{0}))}\beta w_{\alpha}^{2}d\mathcal{H}^{N-1}<\infty
\]
due to \eqref{bound-b}. Thus there exists a positive constant $C$ independent of all large $\alpha$, such that
\begin{equation}\label{in-0}
\int_{B'(0',r_{0})}\tilde{w}_{\alpha}^{2}dx'\leq C.
\end{equation}
We choose $r>0$ sufficiently small such that
$Z(B(0,r/2)\cap\Omega)\subset B'(0,r)\times(0,r)$.
By the assumption that $\mu(\{0\})>0$, there
exists a sequence of $\{\alpha\}$, denoted by itself for convenience, such that
\[
  \int_{B(0,r/2)\cap\Omega}w_{\alpha}^{2}dx>\frac{\mu(\{0\})}{2},
\]
which yields
\[
\int_{B'(0',r)\times(0,r)}\tilde{w}_{\alpha}^{2}dz\geq \frac{\mu(\{0\})}{2}.
\]
Thus there exists $\tau\in(0,r)$ such that
\begin{equation}\label{in-1}
\int_{B'(0',r)\times\{\tau\}}\tilde{w}_{\alpha}^{2}dz'\geq \frac{\mu(\{0\})}{2r}.
\end{equation}
In addition, because of $\partial_{n}m(0)=-\partial_{z_{N}}\tilde{m}(0)>0$, we may
choose $r\in(0,r_{0})$ small enough such that
\begin{equation}\label{in-2}
\partial_{z_{N}}\tilde{m}(z)<0\ \ \text{ in }\, B'(0',r)\times[0,r].
\end{equation}
For any given $f$,  let $\tilde{f}(z)=f(X(z))=f(x)$. Then we have
$$
\mbox{$\partial_{z_{k }}f(z)=\partial_{x_{k}}f(x)+\partial_{x_{N}}f(x)\partial_{x_{k}}\psi(x')$, \ \ for\, $k=1,2,\cdots, N-1,$}
$$
and
$\partial_{Z_{N}}\tilde{f}(z)=\partial_{x_{N}}f(x)$. Thus
\begin{equation}\label{in-2-1}
  \begin{aligned}
  |\nabla\tilde{f}(z)|&=|\nabla
  f(x)+\partial_{x_{N}}f(x)(\nabla_{x'}\psi(x'),0)|\\
  &\leq |\nabla
  f(x)|+|\partial_{x_{N}}f(x)(\nabla_{x'}\psi(x'),0)|\\
  &\leq 2|\nabla f(x)|,
  \end{aligned}
  \end{equation}
  where $\nabla_{x'}\psi=(\partial_{x_{1}}\psi,\cdots,\partial_{x_{N-1}}\psi)$.

 Using \eqref{in-2} and \eqref{in-2-1}, some calculation shows that
\begin{equation}\label{in-3}
  \begin{aligned}
   & \int_{\Omega}|\nabla w_{\alpha}-\alpha w_{\alpha}\nabla m|^{2}dx\\
   &=\int_{\Omega}e^{2\alpha m}|\nabla \varphi_{\alpha}|^{2}dx\\
   &\geq\frac{1}{2}\int_{B'(0',r)\times(0,r)}e^{2\alpha\tilde m}|\nabla\tilde \varphi_{\alpha}|^{2}dz\\
    &=\frac{1}{2}\int_{B'(0',r)\times(0,r)}|\nabla\tilde{w}_{\alpha}
    -\alpha\tilde{w}_{\alpha}\nabla\tilde{m}|^{2}dz\\
    &\geq\frac{1}{2}\int_{B'(0',r)\times(0,r)}|\partial_{z_{N}}\tilde{w}_{\alpha}
    -\alpha\tilde{w}_{\alpha}\partial_{z_{N}}\tilde{m}|^{2}\chi_{\{\partial_{z_{N}}\tilde{w}_{\alpha}>0\}}dz\\
    &\geq\frac{1}{2}\int_{B'(0',r)\times(0,r)}[(\partial_{z_{N}}\tilde{w}_{\alpha})^{+}]^{2}dz\\
    &=\frac{1}{2}\int_{B'(0',r)}\int_{0}^{r}[(\partial_{z_{N}}\tilde{w}_{\alpha})^{+}]^{2}dz_{N}dz'\\
   & \geq
   \frac{1}{2}\int_{B'(0',r)}\left(\int_{0}^{r}1^{2}dz_{N}\right)^{-1}\left(\int_{0}^{r}
   (\partial_{z_{N}}\tilde{w}_{\alpha})^{+}dz_{N}\right)^{2}dz'\\
&= \frac{1}{2r}\int_{B'(0',r)}\left(\int_{0}^{r}(\partial_{z_{N}}\tilde{w}_{\alpha})^{+}dz_{N}\right)^{2}dz'.
 \end{aligned}
\end{equation}

Let $\mathcal{Q}=\{x'\in B'(0',r):\ \ \tilde{w}_{\alpha}(x',\tau)\geq
\tilde{w}_{\alpha}(x',0)\}$. Then for each $x'\in \mathcal{Q}$, we have
\[
\tilde{w}_{\alpha}(x',\tau)-
\tilde{w}_{\alpha}(x',0)=\int_{0}^{\tau}\partial_{z_{N}}\tilde{w}_{\alpha}(x',z_{N})dz_{N}\leq \int_{0}^{\tau}\left(\partial_{z_{N}}\tilde{w}_{\alpha}\right)^{+}(x',z_{N})dz_{N}.
\]
Thus, it holds that
 \begin{equation}\label{in-4}
\begin{aligned}
  \int_{\mathcal{Q}}\left(\int_{0}^{r}(\partial_{z_{N}}\tilde{w}_{\alpha})^{+}dz_{N}\right)^{2}dz'
  &\geq\int_{\mathcal{Q}}\left(\tilde{w}_{\alpha}(z',\tau)\right)^{2}
  -2\tilde{w}_{\alpha}(z',\tau)\tilde{w}_{\alpha}(z',0)+\left(\tilde{w}_{\alpha}(z',0)\right)^{2}dz'\\
  &\geq\int_{\mathcal{Q}}\frac{1}{2}\left(\tilde{w}_{\alpha}(z',\tau)\right)^{2}
  -3\left(\tilde{w}_{\alpha}(z',0)\right)^{2}dz'\\
   &= \frac{1}{2}\int_{\mathcal{Q}}\left(\tilde{w}_{\alpha}(z',\tau)\right)^{2}
   -6\left(\tilde{w}_{\alpha}(z',0)\right)^{2}dz'.
 \end{aligned}
\end{equation}

On the other hand, for $x'\in B'(0',r)\setminus \mathcal{Q}$, due to
\[
 \tilde{w}_{\alpha}(x',\tau)<
\tilde{w}_{\alpha}(x',0),
\]
  we have
  \begin{equation}\label{in-5}
    \int_{B'(0',r)\setminus
      \mathcal{Q}}\left(\tilde{w}_{\alpha}(z',\tau)\right)^{2}
      -6\left(\tilde{w}_{\alpha}(z',0)\right)^{2}dz\leq 0.
  \end{equation}

Therefore, it follows from \eqref{in-4} and \eqref{in-5} that
 \begin{equation*}
\begin{aligned}
  \int_{B'(0',r)}\left(\int_{0}^{r}(\partial_{z_{N}}\tilde{w}_{\alpha})^{+}dz_{N}\right)^{2}dz'
 & \geq
  \int_{\mathcal{Q}}\left(\int_{0}^{r}(\partial_{z_{N}}\tilde{w}_{\alpha})^{+}dz_{N}\right)^{2}dz' \\
  &\geq
  \frac{1}{2}\int_{\mathcal{Q}}\left(\tilde{w}_{\alpha}(z',\tau)\right)^{2}
  -6\left(\tilde{w}_{\alpha}(z',0)\right)^{2}dz'\\
   &\geq
   \frac{1}{2}\int_{B'(0',r)}\left(\tilde{w}_{\alpha}(z',\tau)\right)^{2}
   -6\left(\tilde{w}_{\alpha}(z',0)\right)^{2}dz'\\
   &=
   \frac{1}{2}\left(\int_{B'(0,r)\times\{\tau\}}\tilde{w}_{\alpha}^{2}dz'
   -6\int_{B'(0,r)\times\{0\}}\tilde{w}_{\alpha}^{2}dz'\right),
     \end{aligned}
\end{equation*}
and consequently, thanks to \eqref{in-0} and \eqref{in-1}, we obtain
from \eqref{in-3} that, for any given $M<\infty$,
 \begin{equation*}
   \begin{aligned}
     \int_{\Omega}|\nabla w_{\alpha}-\alpha m\nabla
     w_{\alpha}|^{2}dx&\geq
    \frac{1}{2r} \int_{B'(0',r)}\left(\int_{0}^{r}(\partial_{z_{N}}\tilde{w}_{\alpha})^{+}dz_{N}\right)^{2}dz' \\
   &\geq \frac{1}{4r}\left(\frac{\mu(\{0\})}{2r}-6C\right)\\
   &\geq M,
 \end{aligned}
\end{equation*}
provided that $r$ is sufficiently small and $\alpha$ is sufficiently large, which
contradicts with \eqref{bound-a}. Thus we can conclude that $\mu(\{0\})=0$.

\vskip5pt \textbf{Case 2.} $|\nabla m(0)|=0$ and $(e\cdot\nabla)^{2}m(0)<0$ for all
$e\in S^{N-1}$. We shall utilize a similar technique as in Case 1 to produce a contradiction.

Since $|\nabla m(0)|=0$ and $\det(D^{2}m(0))\neq 0$, by \cite[Lemma 3.1]{CL2012}, there
exist a constant $r_{0}>0$ and a smooth vector field
$$\mbox{$\tau(x):\ \ B(0,r_{0})\to\mathbb{R}^{N}$ such that $\tau\cdot n>0$ and
$\tau\cdot\nabla m=0$ on $\partial\Omega\cap B(0,r_{0})$.}$$
We may assume that $|\tau|\equiv 1$.
Since
\[
(\tau(0)\cdot\nabla)^{2} m(0)<0,
\]
there exists $r_{1}\in(0,r_{0})$ such that
\[
(\tau(x)\cdot\nabla)^{2}m(x)<0,\ \ \, \forall x\in\Omega\cap
B(0,r_{1}),
\]
from which it then follows that
\[
\tau(x)\cdot\nabla m(x)<0,\ \ \ \forall x\in\Omega\cap
B(0,r_{1}).
\]

Let
\[
  \Omega_{t}=\{x\in\Omega:\ \ \text{dist}(x,\partial\Omega)> t\}.
\]
Since
$\partial\Omega$ is smooth, $\partial\Omega_{t}$ is an $(N-1)$-dimensional smooth
manifold if $t$ is small enough. For $x\in\partial\Omega_{t}$, we still
use $n(x)$ to denote the unit outward normal vector of $\Omega_{t}$ at
$x$.  Since $\tau(x)\cdot n(x)>0$ for $x\in\partial\Omega\cap
B(0,r_0)$, there exist constants $\tau_0$ and $r_{2}\in(0,r_{1})$ such that
\[
  \inf_{x\in
  \Omega\cap B(0,r_2)}\tau(x)\cdot n(x)\geq
\tau_0>0.
\]

For any given $x\in\partial\Omega\cap B(0,r_{2})$, we denote
$\xi(x,t)$ to be the solution of the ordinary differential system
\[
  \frac{d\xi(t)}{dt}=-\tau(\xi(t)),\ \ \ \xi(0)=x.
  \]
  For each $x\in\partial\Omega\cap B(0,r_{2})$ and $t\in(0,r_{2})$,
  there exists $p(x,t)$ such that
  $\xi(x,p(x,t))\in\partial\Omega_{t}$. Moreover,
 \begin{equation}\label{in-06}
    \frac{\partial}{\partial t}p(x,t)=(\tau(x)\cdot
    n(x))^{-1}\in[1,\tau_{0}^{-1}].
  \end{equation}
Let  $r_{3}\in(0,r_{2})$. For each $t\in[0,r_{3})$, we define
\[
\Gamma_{t}=\{\xi(x,p(x,t)):\ \ x\in\partial\Omega\cap B(0,r_{3})\}.
\]
Then $\Gamma_{t}\subset\partial\Omega_{t}$ is an $(N-1)$-dimensional
smooth manifold. We may choose $r_{3}$ small enough such that
$\Gamma_{t}\subset B(0,r_{2})$ for all $t\in (0,r_{3})$. Clearly,
for sufficiently small $r_{4}\in(0,r_{3})$,
 \begin{equation}\label{add-6}
  \overline{\Omega}\cap B(0,r_{4})\subset
  \cup_{t\in[0,r_{3}]}\Gamma_{t}.
  \end{equation}
For each $t\in(0,r_{3})$,
\[
  \Phi^{t}:\Gamma_{0}\to\Gamma_{t},\ \ \ \Phi^{t}(x)=\xi(x,p(x,t))
  \]
is a diffeomorphism. There exists a constant $\sigma>0$ such that
  \begin{equation}\label{pushforward}
\sigma\mathcal{H}^{N-1}|_{\Gamma_{t}}\leq\Phi^{t}_{*}\left(\mathcal{H}^{N-1}|_{\Gamma_{0}}\right)\leq \sigma^{-1} \mathcal{H}^{N-1}|_{\Gamma_{t}}
    \end{equation}
    for all $t\in(0,r_{3})$. Here,
    $\Phi^{t}_{*}\left(\mathcal{H}^{N-1}|_{\Gamma_{0}}\right)$ is
    the push-forward measure of $\mathcal{H}^{N-1}$ restricted to
    $\Gamma_{0}$.

Using the formula of integration over level sets (see
e.g. \cite[Theorem 3.13]{EG2015}), we have
  \begin{equation}\label{add-7}
  \int_{\cup_{t\in[0,\delta]}\Gamma_{t}}f(x)dx=\int_{0}^{\delta}\int_{\Gamma_{t}}f(x)d\mathcal{H}^{N-1}dt,
    \end{equation}
for all $f\in L^{1}(\Omega)$.

As in Case 1, towards a contradiction we suppose that $\mu(\{0\})>0$ and \eqref{bound-b} holds. Since $\beta(0)>0$, we may choose $r_{4}\in(0,r_{3})$ so small that $\beta(x)\geq
\frac{\beta(0)}{2}$ for all $x\in\partial\Omega\cap B(0,r_{4})$. Thus, by \eqref{bound-b}, there exists a constant $C$, independent of all $\alpha\geq1$, such that
 \begin{equation}\label{in-6}
\int_{\Gamma_{0}}w_{\alpha}^{2}d\mathcal{H}^{N-1}\leq C,\ \ \ \forall \alpha\geq1.
    \end{equation}

On the other hand, by virtue of $\mu(\{0\})>0$, we may assume
\[
\int_{B(0,r_4)\cap\Omega}w_{\alpha}^{2}dx>\frac{\mu(\{0\})}{2}
\]
for all large $\alpha$. Combined with \eqref{add-6} and \eqref{add-7},
this then gives
\[
\int_{\cup_{t\in[0,r_{4}]}\Gamma_{t}}w_{\alpha}^{2}dx
=\int_{0}^{r_{4}}\int_{\Gamma_{t}}w_{\alpha}^{2}d\mathcal{H}^{N-1}dt\geq \int_{\Omega\cap
  B(0,r_4)}w_{\alpha}^{2}dx\geq \frac{\mu(\{0\})}{2}>0,
\]
which allows us to find $\delta\in(0,r_{4})$ such that
 \begin{equation}\label{in-7}
\int_{\Gamma_{\delta}}w_{\alpha}^{2}d\mathcal{H}^{N-1}\geq \frac{\mu(\{0\})}{2r_{4}}.
\end{equation}

Similarly to Case 1, since $\tau\cdot \nabla m<0$ and $|\tau|\equiv 1$ in $\Omega\cap B(0,r_{1})$,  we have
  \begin{equation}\label{in-8}
\begin{aligned}
    \int_{\Omega}|\nabla w_{\alpha}-\alpha w_{\alpha}\nabla m|^{2}dx    &\geq\int_{\cup_{t\in[0,\delta]}\Gamma_{t}}|\nabla
    w_{\alpha}-\alpha w_{\alpha}\cdot\nabla m|^{2}dx\\
    &\geq\int_{\cup_{t\in[0,\delta]}\Gamma_{t}}|\tau\cdot\nabla
    w_{\alpha}-\alpha w_{\alpha}\tau\cdot\nabla m|^{2}dx\\
    &\geq \int_{0}^{\delta}\int_{\Gamma_{t}}[(\tau\cdot\nabla
    w_{\alpha})^{+}]^{2}d\mathcal{H}^{N-1}dt.
    \end{aligned}
  \end{equation}
On the other hand, it follows from \eqref{pushforward} that
  \[
\sigma\int_{\Gamma_{0}}|u(\Phi^{t}(x))|^{2}d\mathcal{H}^{N-1}\leq
\int_{\Gamma_{t}}|u(x)|^{2}d\mathcal{H}^{N-1}\leq \sigma^{-1}\int_{\Gamma_{0}}|u(\Phi^{t}(x))|^{2}d\mathcal{H}^{N-1},
\]
for all $u\in W^{1,2}(\Omega)$. Thus, in light of \eqref{in-06} and \eqref{in-8}, we obtain
 \begin{equation}\label{in-9}
   \begin{aligned}
  \int_{\Omega}|\nabla w_{\alpha}-\alpha w_{\alpha}\nabla m|^{2}dx &\geq \sigma\int_{0}^{\delta}\int_{\Gamma_{0}}[(\tau\cdot\nabla
  w_{\alpha})^{+}]^{2}(\xi(x,p(x,t)))d\mathcal{H}^{N-1}dt\\
 &= \sigma\int_{\Gamma_{0}}\int_{0}^{\delta}[(\tau\cdot\nabla
    w_{\alpha})^{+}]^{2}(\xi(x,p(x,t)))dtd\mathcal{H}^{N-1}\\
    &= \sigma\int_{\Gamma_{0}}\int_{0}^{\delta}[(\partial_{t}w_{\alpha}(\xi(x,p(x,t)))^{+}\left(\partial_{t}p(x,t)\right)^{-1}]^{2}dtd\mathcal{H}^{N-1}\\
    &\geq  \sigma
    \tau_{0}^{2}\int_{\Gamma_{0}}\int_{0}^{\delta}[(\partial_{t}w_{\alpha}(\xi(x,p(x,t))))^{+}]^{2}dtd\mathcal{H}^{N-1}\\
    &\geq \sigma
    \tau_{0}^{2}\int_{\Gamma_{0}}\left(\int_{0}^{\delta}1^{2}dt\right)^{-1}\left(\int_{0}^{\delta}[\partial_{t}w_{\alpha}(\xi(x,p(x,t)))]^{+}dt\right)^{2}d\mathcal{H}^{N-1}\\
    &\geq \frac{\sigma
    \tau_{0}^{2}}{\delta}\int_{\Gamma_{0}}\left(\int_{0}^{\delta}[\partial_{t}w_{\alpha}(\xi(x,p(x,t)))]^{+}dt\right)^{2}d\mathcal{H}^{N-1}.
\end{aligned}
\end{equation}

Denote $\mathcal{Q}=\{x:\ \ x\in \Gamma_{0},\, w_{\alpha}(\xi(x,p(x,\delta)))\geq
w_{\alpha}(\xi(x,p(x,0)))\}$. Similarly to Case 1, we have
\begin{equation}\label{in-9-1}
  \begin{aligned}
    \int_{\Gamma_{0}}\left(\int_{0}^{\delta}[\partial_{t}w_{\alpha}(\xi(x,p(x,t)))]^{+}dt\right)^{2}d\mathcal{H}^{N-1}&\geq
    \int_{\mathcal{Q}}\left(\int_{0}^{\delta}[\partial_{t}w_{\alpha}(\xi(x,p(x,t)))]^{+}dt\right)^{2}d\mathcal{H}^{N-1}\\
    &\geq
    \int_{\mathcal{Q}}\left(w_{\alpha}(\xi(x,p(x,\delta))-w_{\alpha}(\xi(x,0)\right)^{2}d\mathcal{H}^{N-1}\\
    &\geq
    \int_{\mathcal{Q}}\frac{1}{2}(w_{\alpha}(\xi(x,p(x,\delta)))^{2}-3(w_{\alpha}(\xi(x,0))^{2}d\mathcal{H}^{N-1}\\
      &\geq
    \frac{1}{2}\int_{\Gamma_{0}}(w_{\alpha}(\xi(x,p(x,\delta)))^{2}-6(w_{\alpha}(\xi(x,0))^{2}d\mathcal{H}^{N-1}.
  \end{aligned}
\end{equation}
As a result, using \eqref{in-6}-\eqref{in-9-1}, we infer that
 \begin{equation}\label{ineq3}
   \begin{aligned}
 \int_{\Omega}|\nabla w_{\alpha}-\alpha w_{\alpha}\nabla m|^{2}dx
&\geq \frac{\sigma
    \tau_{0}^{2}}{\delta}\int_{\Gamma_{0}}\left(\int_{0}^{\delta}[\partial_{t}w_{\alpha}(\xi(x,p(x,t)))]^{+}dt\right)^{2}d\mathcal{H}^{N-1}\\
    &\geq \frac{
      \sigma\tau_{0}^{2}}{2\delta}\int_{\Gamma_{0}}(w_{\alpha}(\xi(x,p(x,\delta)))^{2}-6(w_{\alpha}(\xi(x,0))^{2}d\mathcal{H}^{N-1}\\
     &\geq \frac{\sigma^{2}\tau_{0}^{2}}{2\delta}\int_{\Gamma_{\delta}}w_{\alpha}^{2}d\mathcal{H}^{N-1}-
    \frac{6\sigma\tau_{0}^{2}}{\delta}\int_{\Gamma_{0}}w_{\alpha}^{2}d\mathcal{H}^{N-1}\\
    &\geq \frac{\sigma^{2}\tau_{0}^{2}\mu(\{0\})}{4\delta
      r_{4}}-\frac{6\sigma\tau_{0}^{2}C}{\delta}.
\end{aligned}
\end{equation}

For any given $r_{4}$, the inequality \eqref{ineq3} holds for
all large $\alpha$. Thus we can choose $r_{4}$ sufficiently small such that
\[
 \frac{\sigma^{2}\tau_{0}^{2}\mu(\{0\})}{4\delta
   r_{4}}-\frac{6\sigma\tau_{0}^{2}C}{\delta}>C_*,
\]
which contradicts with \eqref{bound-a}. Therefore we can assert
that $\mu(\{0\})=0$ and so Theorem \ref{nonnegative}(i) holds.

\vskip6pt
We next verify Theorem \ref{nonnegative}(ii). We argue indirectly again by supposing that
  \begin{equation}\label{in-10}
   \liminf_{\alpha\to\infty}\lambda_1(\alpha)\leq M,
\end{equation}
 for some constant $M$. Then there exists a sequence of $\alpha$ which converges to $\infty$, denoted by itself for notational convenience, such that $\lambda_1(\alpha)\leq M+1$ for all $\alpha\geq0$. This, together with \eqref{in-10}, implies that
 \[
   \int_{\Omega}|\nabla w_{\alpha}-\alpha w_{\alpha}\nabla
   m|^{2}dx+\int_{\partial\Omega}\beta w_{\alpha}^{2}d\mathcal{H}^{N-1}\leq
   M+1+\max_{x\in\overline\Omega}V(x)
\]
for all large $\alpha$. Similarly to the argument in (i), under the assumption $\Sigma_{1}\cup\Sigma_{2}=\emptyset$, we are able to show that
\[
  \text{supp}(\mu)\cap\overline{\Omega}=\emptyset,
\]
which leads to $\mu(\overline{\Omega})=0$, contradicting with $\mu(\overline{\Omega})=1$. Therefore
one must have
\[
  \lim_{\alpha\to\infty}\lambda_1(\alpha)=\infty.
\]
This proves Theorem \ref{nonnegative}(ii).
\end{proof}

\subsection{Proof of Theorem \ref{Dirichlet}} In this subsection, we consider the eigenvalue problem (\ref{eq:1.1-d}). It is known that the principal eigenvalue of (\ref{eq:1.1-d}) can be characterized by
\begin{equation}\label{eq:3}
  \begin{split}
    {\lambda}(\alpha)&=\inf_{\varphi\in
                     H_{0}^{1}(\Omega)}\frac{\int_{\Omega}e^{2\alpha
                     m}(|\nabla\varphi|^{2}+V\varphi^{2})dx}{\int_{\Omega}e^{2\alpha
                     m}\varphi^{2}dx}\\
&=\inf_{w\in
                   H_{0}^{1}(\Omega),\,\int_{\Omega}w^{2}dx=1}\int_{\Omega}|\nabla
                   w-\alpha w\nabla
                   m|^{2}+Vw^{2}dx.
\end{split}
\end{equation}
Indeed, the second variational characterization in \eqref{eq:3} is derived through
the substitution $\varphi=e^{-\alpha m}w$ in (\ref{eq:1.1-d}). Obviously, $w=e^{\alpha m}\varphi$
solves
 \begin{equation}
  \nonumber
\begin{cases}
  -\Delta w+\left(\alpha^2|\nabla m|^2+\alpha\Delta m+V\right)w=\lambda_1(\alpha)w &\hbox{ in }\Omega,\\
  w=0 &\hbox{ on }\partial\Omega.
\end{cases}
\end{equation}

We now present the proof of Theorem \ref{Dirichlet}.

\begin{proof}[Proof of Theorem \ref{Dirichlet}] We consider the eigenvalue problem
  \eqref{eq:1.1} with $\beta\equiv 1$, and denote the corresponding principal eigenvalue by
  $\lambda_1(\alpha)$. Also for notational clarity, we denote by $\tilde {\lambda}_1(\alpha)$ the principal eigenvalue of the eigenvalue problem \eqref{eq:1.1-d}. Then, it is well known that
  \[
\tilde{\lambda}_1(\alpha)\geq \lambda_1(\alpha),\ \  \ \forall \alpha>0.
\]
Observe that $\Sigma_{2}=\emptyset$ due to $\beta\equiv 1$. Thus, if $\Sigma_{1}=\emptyset$, Theorem \ref{Dirichlet}(ii) readily follows from Theorem \ref{nonnegative}(ii).

If $\Sigma_{1}\not=\emptyset$, we infer from Theorem \ref{nonnegative}(i) that
 \begin{equation}
  \label{eq:3a}
\liminf_{\alpha\to\infty}\tilde{\lambda}_1(\alpha)\geq\liminf_{\alpha\to\infty}{\lambda_1}(\alpha)
=\min_{x\in\Sigma_{1}}V(x).
\end{equation}

Let $x_{0}\in\Sigma_{1}$ and $\varphi_{j}(x)$ be the functions defined
in the proof of Lemma \ref{upperbound}. Clearly $\varphi_{j}(x)=0$ on
$\partial\Omega$. Substituting $\varphi_{j}$ into the second expression of \eqref{eq:3}, we obtain, similarly to the proof of Lemma \ref{upperbound}, that
\[
  \limsup_{\alpha\to\infty}\tilde{\lambda}_1(\alpha)\leq V(x_{0}).
\]
Since $x_{0}$ is arbitrary, we have
 \begin{equation}
  \label{eq:3b}
  \limsup_{\alpha\to\infty}\tilde{\lambda}_1(\alpha)\leq\min_{x\in\Sigma_{1}} V(x).
\end{equation}
Combining \eqref{eq:3a} and \eqref{eq:3b}, we deduce Theorem \ref{Dirichlet}(i).
 \end{proof}

\subsection{Proof of Theorem \ref{general}}
In this subsection, we consider the general case that $\beta(x)$ may change its
sign or be negative.

Given $0<r_0<1$, for each $x\in\partial\Omega^{-}_{r_{0}}=\{x\in \partial\Omega:\ \ \mathrm{dist}(x,\partial\Omega^{-})<r_{0}\}$, let $\eta(x,t)$ be the
 solution of
 \[
   \frac{d\eta}{dt}(x,t)=\frac{-\nu(\eta(x,t))}{\nu(\eta(x,t))\cdot
     n(\eta(x,t))},\ \ \ \eta(x,0)=x.
 \]

 For each $x\in\Omega^{-}_{r_{0}}$ and $t\in(0,r_{0})$,
 $\eta(x,t)\in\partial\Omega_{t}^-$.  Let us define
\[
\Gamma_{t}=\{\eta(x,t):\ \ x\in\partial\Omega_{r_{0}}^{-}\}.
\]
Then the mapping
 \[
\Psi^{t}:\ \ \Gamma_{0}\to \Gamma_{t},\ \ \Psi^{t}(x)=\eta(x,t)
\]
is a diffeomorphism, and there exists a constant $\sigma>0$
such that
 \begin{equation*}
  \sigma\mathcal{H}^{N-1}|_{\Gamma_{t}}\leq \Psi^{t}_{*}\ \left(\mathcal{H}^{N-1}|_{\Gamma_{0}}\right)\leq \sigma^{-1} \mathcal{H}^{N-1}|_{\Gamma_{t}}
\end{equation*}
for all $t\in (0,r_{0})$, where $\Phi^{t}_{*}\left(\mathcal{H}^{N-1}|_{\Gamma_{0}}\right)$ is the push-forward measure of  $\mathcal{H}^{N-1}$ restricted to $\Gamma_{0}$. Thus, for all $f\in W^{1,1}(\Omega)$ and $t\in(0,r_{0})$, it holds that
\begin{equation}\label{eq:1.3}
\sigma\int_{\Gamma_{t}}|f(x)|d\mathcal{H}^{N-1}\leq\int_{\Gamma_{0}}|f(\eta(x,t))|d\mathcal{H}^{N-1}\leq \sigma^{-1}\int_{\Gamma_{t}}|f(x)|d\mathcal{H}^{N-1}.
\end{equation}

In the following, we prepare a trace inequality which will be used later.

\begin{lemma}\label{lemma-trace}
  Let $x_{0}\in\partial\Omega$, $r\in(0,r_{0})$ and $\zeta\in
  C_{0}^{1}(B(x_{0},r))$. Assume that there exist a smooth unit vector field $\nu(x)$ and positive
  constants $\delta$ and $r$ such that $\nu(x)\cdot n(x)\geq\delta$ and $\nu(x)\cdot\nabla
  m(x)\leq 0 $ for all $x\in\Omega\cap B(x_{0},r)$. Then for any given $\epsilon>0$, there
  exists a constant $C(\epsilon)$, which is independent of $\alpha\geq0$, such
  that
  \begin{equation}\label{trace}
    \int_{\partial\Omega\cap B(x_{0},r)}u^{2}\zeta^{2}e^{2\alpha
        m}d\mathcal{H}^{N-1}\leq
      \epsilon\int_{\Omega\cap B(x_{0},r)}|\nabla u|^{2}\zeta^{2}e^{2\alpha
        m}dx+C(\epsilon)\int_{\Omega\cap B(x_{0},r)}u^{2}e^{2\alpha
        m}dx
  \end{equation}
for all $u\in W^{1,2}(\Omega\cap B(x_{0},r))$.
\end{lemma}

\begin{proof} It suffices to prove \eqref{trace} for any smooth function $u$.

For any given $u\in  C^{\infty}(\overline{\Omega\cap B(x_{0},r)})$, in view of \eqref{eq:1.3}, we have
  \begin{equation*}
    \begin{aligned}
      & \int_{\partial\Omega\cap B(x_{0},r)}u^{2}(x)\zeta^{2}(x)e^{2\alpha
        m(x)}d\mathcal{H}^{N-1}\\
      & =\int_{\partial\Omega\cap B(x_{0},r)}\int_{0}^{r}-\frac{d}{dt}\left[u^{2}(\eta(x,t))\zeta^{2}(\eta(x,t))e^{2\alpha
          m(\eta(x,t))}\right]dt d\mathcal{H}^{N-1}\\
       & =\int_{\partial\Omega\cap B(x_{0},r)}\int_{0}^{r}-\frac{d\eta(x,t)}{dt}\cdot\left(2u\nabla u\zeta^{2}e^{2\alpha m}+2u^{2}\zeta\nabla \zeta
      e^{2\alpha m}+2\alpha\nabla m u^{2}\zeta^{2}e^{2\alpha
        m}\right)(\eta(x,t))dtd\mathcal{H}^{N-1}\\
    &=\int_{\partial\Omega\cap B(x_{0},r)}\int_{0}^{r}
    \frac{1}{\nu(\eta(x,t))\cdot n(\eta(x,t))}\nu(\eta(x,t))\cdot\left(2u\nabla u\zeta  ^{2}e^{2\alpha m}+2u^{2}\zeta\nabla \zeta
      e^{2\alpha m}\right)(\eta(x,t))dtd\mathcal{H}^{N-1}\\
    &\ \ \ +2\alpha\int_{\partial\Omega\cap B(x_{0},r)}\int_{0}^{r}
    \frac{\left(u^{2}\zeta  ^{2}e^{2\alpha
        m}\right)(\eta(x,t))}{\nu(\eta(x,t))\cdot n(\eta(x,t))}\nu(\eta(x,t))\cdot\nabla m (\eta(x,t))dtd\mathcal{H}^{N-1}\\
          &\leq\delta^{-1}\int_{\partial\Omega\cap B(x_{0},r)}\int_{0}^{r}
     \left| \nu(x)\cdot\left(2u\nabla u\zeta  ^{2}e^{2\alpha m}+2u^{2}\zeta\nabla \zeta
        e^{2\alpha m}\right)(\eta(x,t))\right|dtd\mathcal{H}^{N-1}\\
      &\leq\delta^{-1}\int_{0}^{r}\int_{\partial\Omega\cap B(x_{0},r)} \left(2|u||\nabla u|\zeta  ^{2}e^{2\alpha m}+2u^{2}|\zeta  ||\nabla \zeta|
      e^{2\alpha m}\right)(\eta(x,t))d\mathcal{H}^{N-1}dt\\
    &\leq \delta^{-1}\sigma^{-1}\int_{0}^{r}\int_{\Gamma_{t}\cap B(x_{0},r)}\left( 2|u||\nabla u|\zeta^{2}e^{2\alpha m}+2u^{2}|\zeta||\nabla \zeta|
    e^{2\alpha m}\right)(x)d\mathcal{H}^{N-1}dt\\
     &=\delta^{-1} \sigma^{-1}\int_{\Omega\cap B(x_{0},r)} 2|u||\nabla u|\zeta^{2}e^{2\alpha m}+2u^{2}|\zeta||\nabla \zeta|
      e^{2\alpha m}dx\\
      &\leq \epsilon\int_{\Omega\cap B(x_{0},r)}|\nabla u|^{2}\zeta^{2}e^{2\alpha
        m}dx+\frac{1}{\epsilon\delta \sigma}\int_{\Omega\cap B(x_{0},r)}u^{2}\zeta^{2}e^{2\alpha
        m}dx\\
      &\ \ \ +2\delta^{-1}\sigma^{-1}\max_{\overline{\Omega\cap B(x_{0},r)}}|\zeta||\nabla\zeta|\int_{\Omega\cap B(x_{0},r)}u^{2}e^{2\alpha
        m}dx\\
      &\leq \epsilon\int_{\Omega\cap B(x_{0},r)}|\nabla u|^{2}\zeta^{2}e^{2\alpha
        m}dx+C(\epsilon)\int_{\Omega\cap B(x_{0},r)}u^{2}e^{2\alpha m}dx.
    \end{aligned}
  \end{equation*}
This completes the proof.
\end{proof}

With the help of Lemma \ref{lemma-trace}, we are able to obtain the following boundedness of $\{\lambda_1(\alpha)\}$ from below as $\alpha\to\infty$.

\begin{lemma}\label{lowerbound} Assume that there exist a smooth unit vector field $\nu(x)$ and positive
  constants $\delta$ and $r_{0}$ such that $\nu(x)\cdot n(x)\geq\delta$ and $\nu(x)\cdot\nabla
  m(x)\leq 0 $ for all $x\in\Omega_{r_{0}}^{-}$. Then we have
  \[
    \liminf_{\alpha\to\infty}\lambda_1(\alpha)>-\infty.
  \]
\end{lemma}

\begin{proof}
  Let $r\in (0,\delta)$. We choose a set
  $\{x_k\in\partial\Omega^{-}:\ 1\leq k\leq K\}$ consisting of finitely many points and a sequence of smooth functions
  $\{\zeta_{k}:\ 1\leq k\leq K\}$ such that
$$\mbox{
  $\partial\Omega^{-}\subset\bigcup_{k=1}^{K}B(x_{k},r)$,\ \ \
  ${\zeta_{k}}=0$\ \, in $\mathbb{R}^{N}\setminus B(x_{k},r)$,\ \ \
  $\sum_{k=1}^{K}\zeta_{k}^{2}\leq 1$}
  $$ and
$$\mbox{$\sum_{k=1}^{K}\zeta_{k}^{2}(x)=1$\ \ \ for all\ $x\in \partial\Omega^{-}$.}
$$
  For any $\epsilon>0$, by Lemma \ref{lemma-trace}, there exist
  constants $C_{k}(\epsilon)$ such that, for all $\varphi\in
  W^{1,2}(\Omega)$,
\begin{equation*}
    \int_{\partial\Omega\cap B(x_{k},r)}\varphi^{2}\zeta_{k}^{2}e^{2\alpha
        m}d\mathcal{H}^{N-1}\leq
      \epsilon\int_{\Omega\cap B(x_{k},r)}|\nabla \varphi|^{2}\zeta_{k}^{2}e^{2\alpha
        m}dx+C_{k}(\epsilon)\int_{\Omega\cap B(x_{k},r)}\varphi^{2}e^{2\alpha
        m}dx.
  \end{equation*}
Thus, it holds that
 \begin{equation}
   \begin{aligned}
\int_{\partial\Omega^{-}}\varphi^{2}e^{2\alpha m}d\mathcal{H}^{N-1}&\leq\int_{\partial\Omega^{-}}\sum_{k=1}^{K}\varphi^{2}\zeta_{k}^{2}e^{2\alpha m}d\mathcal{H}^{N-1}\\
&\leq \sum_{k=1}^{K}\left(\epsilon\int_{\Omega}|\nabla\varphi|^{2}\zeta_{k}^{2}e^{2\alpha
  m}dx+C_{k}(\epsilon)\int_{\Omega}\varphi^{2}e^{2\alpha m}dx\right)\\
&\leq \epsilon\int_{\Omega}|\nabla\varphi|^{2} e^{2\alpha
  m}dx+C(\epsilon)\int_{\Omega}\varphi^{2}e^{2\alpha m}dx
\end{aligned}
\label{inq-cc}
\end{equation}
for some constant $C(\epsilon)$.

Let $\beta^{*}=\sup_{\partial\Omega^{-}}|\beta(x)|$ and choose
a small $\epsilon$ such that $\epsilon \beta^{*}\leq 1$. Then for all
$\varphi\in W^{1,2}(\Omega)$, it holds that
\begin{equation*}
  \begin{aligned}
    &\frac{\int_{\Omega}e^{2\alpha
        m}(|\nabla\varphi|^{2}+V\varphi^{2})dx+\int_{\partial\Omega}\beta
      e^{2\alpha m}\varphi^{2}d\mathcal{H}^{N-1}}{\int_{\Omega}e^{2\alpha
        m}\varphi^{2}dx}\\
    & \geq \frac{\int_{\Omega}e^{2\alpha
        m}(|\nabla\varphi|^{2}+V\varphi^{2})dx-\beta^{*}\left(\epsilon\int_{\Omega}|\nabla\varphi|^{2}e^{2\alpha
          m}dx+C(\epsilon)\int_{\Omega}e^{2\alpha
          m}\varphi^{2}dx\right)}{\int_{\Omega}e^{2\alpha
        m}\varphi^{2}}\\
    &\geq \frac{\int_{\Omega}V e^{2\alpha
        m}\varphi^{2}-\beta^{*}C(\epsilon)e^{2\alpha
        m}\varphi^{2}dx}{\int_{\Omega}e^{2\alpha
        m}\varphi^{2}dx}\\
    &\geq \min_{\overline{\Omega}}V(x)-\beta^{*}C(\epsilon).
  \end{aligned}
\end{equation*}
Therefore, we have
\[
  \lambda_1(\alpha) =\inf_{\varphi\in
    W^{1,2}(\Omega)}\frac{\int_{\Omega}e^{2\alpha
      m}(|\nabla\varphi|^{2}+V\varphi^{2})dx+\int_{\partial\Omega}\beta
    e^{2\alpha m}\varphi^{2}d\mathcal{H}^{N-1}}{\int_{\Omega}e^{2\alpha
      m}\varphi^{2}dx}\geq
  \min_{\overline{\Omega}}V(x)-\beta^{*}C(\epsilon)
\]
for all $\alpha>0$, and consequently,
\[
  \liminf_{\alpha\to\infty} \lambda_1(\alpha)\geq \min_{\overline{\Omega}}V(x)-\beta^{*}C(\epsilon).
\]
This completes the proof.
\end{proof}

Now we are ready to prove Theorem \ref{general}.

\begin{proof}[Proof of Theorem \ref{general}] The assertion (i) follows directly from Lemma \ref{unbounded}.

 In the sequel, we will prove the assertion (ii) under conditions (A1) and (A2). By Lemma \ref{lowerbound}, we have
  \[
 \liminf_{\alpha\to\infty}\lambda_1(\alpha)>-\infty.
\]
Since $\Sigma_{1}\cup\Sigma_{2}\neq\emptyset$, due to Lemma \ref{upperbound}, we see that
\begin{equation}\label{ineq-z1}
  \limsup_{\alpha\to\infty} \lambda_1(\alpha)\leq \min_{\Sigma_{1}\cup\Sigma_{2}}V(x).
\end{equation}
It remains to show that
\[
  \liminf_{\alpha\to\infty} \lambda_1(\alpha)\geq \min_{\Sigma_{1}\cup\Sigma_{2}}V(x).
\]

As in Subsection 2.2, let us  take $\varphi_{\alpha}$ to be the eigenfunctions corresponding to the
  principal eigenvalues $\lambda_1(\alpha)$ and set $w_{\alpha}=e^{\alpha
    m}\varphi_{\alpha}$ with $\int_{\Omega}w_{\alpha}^{2}dx=1$. We may assume that
  $w_{\alpha}^{2}$ weakly converges to a Radon measure $\mu$ in the sense of \eqref{weakcon} with
  $\mu(\overline{\Omega})=1$.
Noticing that, by \eqref{ineq-z1}, for all large $\alpha$,
\begin{equation}\label{ineq-z3}
\lambda_1(\alpha)=\int_{\Omega}(|\nabla\varphi_{\alpha}|^{2}+V\varphi_{\alpha}^{2})e^{2\alpha
m}dx+\int_{\partial\Omega}\beta\varphi_{\alpha}^{2}e^{2\alpha m}d\mathcal{H}^{N-1}\leq \min_{\Sigma_{1}\cup\Sigma_{2}}V(x)+1.
\end{equation}
On the other hand, similarly as to derive \eqref{inq-cc}, one can use Lemma \ref{lemma-trace} to claim
\begin{equation}\label{ineq-z4}
\int_{\partial\Omega}\beta^-\varphi_{\alpha}^{2}e^{2\alpha m}d\mathcal{H}^{N-1}\leq\frac{1}{2}\int_{\Omega}|\nabla\varphi_{\alpha}|^{2} e^{2\alpha
  m}dx+C
\end{equation}
for some constant $C(\epsilon)$ independent of $\alpha$. Hence, by \eqref{ineq-z3} and \eqref{ineq-z4}, one can find a constant $M$, independent of all $\alpha\geq0$, such that
  \[
\int_{\Omega}|\nabla\varphi_{\alpha}|^{2}e^{2\alpha
m}dx+\int_{\partial\Omega}\beta^{+}\varphi_{\alpha}^{2}e^{2\alpha m}d\mathcal{H}^{N-1}\leq M,
\]
leading to
\begin{equation}\label{ineq4}
\int_{\Omega}|\nabla w_{\alpha}-\alpha
w_{\alpha}\nabla m|^{2}dx\leq M.
\end{equation}
As before, it follows from \cite{CL2008} that the inequality \eqref{ineq4} implies that
$\mu(\Omega_{i})=0$ for $i=1,\cdots,6$.

Furthermore, in view of \eqref{ineq4}, we can use the same argument as in
the proof of Theorem \ref{nonnegative} to claim that $\mu(\Sigma_{4})=0$.
As a result, we have
\begin{equation}\label{set}
\text{supp}(\mu)\subset\Sigma_{1}\cup\Sigma_{2}.
\end{equation}

Let  $\delta>0$ such that $\Omega_{\delta}^{-}\cap(\Sigma_{1}\cup\Sigma_{2})=\emptyset$.
We choose  $r\in (0,\delta/2)$ sufficiently small, the set
  $\{x_k\in\partial\Omega^{-}:\ 1\leq k\leq K\}$ and a sequence of smooth functions
  $\{\zeta_{k}:\ 1\leq k\leq K\}$ such that
$$
 \mbox{ $\partial\Omega^{-}\subset\cup_{k=1}^{K}B(x_{k},r)$,\ \
  ${\zeta_{k}}=0$\ \, in $\mathbb{R}^{N}\setminus B(x_{k},r)$,\ \
  $\sum_{k=1}^{K}\zeta_{k}^{2}\leq 1$}
  $$
and
$$
\mbox{$\sum_{k=1}^{K}\zeta_{k}^{2}(x)=1$\ \, for all\, $x\in \partial\Omega^{-}$.}
$$
Let $\eta\in C_{0}^{\infty}(\mathbb{R}^{N})$ be a nonnegative function such that
$$
\mbox{$\eta\equiv 0$\ \ in\
$\Omega\setminus\Omega_{\delta}^{-}$\ \ and\ \ $\eta(x)\equiv 1$ in $\Omega_{r}^{-}$.}
$$
For any $\epsilon>0$, by Lemma \ref{lemma-trace}, there exist
  constants $C_{k}(\epsilon)$ such that, for all $\alpha\geq0$,
  \begin{equation*}
    \begin{aligned}
    \int_{\partial\Omega\cap B(x_{k},r)}\varphi_{\alpha}^{2}\zeta_{k}^{2}e^{2\alpha
        m}d\mathcal{H}^{N-1}&\leq
      \epsilon\int_{\Omega\cap B(x_{k},r)}|\nabla \varphi_{\alpha}|^{2}\zeta_{k}^{2}e^{2\alpha
        m}dx+C_{k}(\epsilon)\int_{\Omega\cap B(x_{k},r)}\varphi_{\alpha}^{2}e^{2\alpha
        m}dx\\
      &\leq
      \epsilon\int_{\Omega\cap B(x_{k},r)}|\nabla \varphi_{\alpha}|^{2}\zeta_{k}^{2}e^{2\alpha
        m}dx+C_{k}(\epsilon)\int_{\Omega}\eta\varphi_{\alpha}^{2}e^{2\alpha
        m}dx.
      \end{aligned}
  \end{equation*}
Thus, we obtain
 \begin{equation}\label{ineq5}
   \begin{aligned}
\int_{\partial\Omega^{-}}\varphi_{\alpha}^{2}e^{2\alpha m}d\mathcal{H}^{N-1}&\leq\int_{\partial\Omega^{-}}\sum_{k=1}^{K}\varphi_{\alpha}^{2}\zeta_{k}^{2}e^{2\alpha m}d\mathcal{H}^{N-1}\\
&\leq \sum_{k=1}^{K}\left(\epsilon\int_{\Omega}|\nabla\varphi_{\alpha}|^{2}\zeta_{k}^{2}e^{2\alpha
  m}dx+C_{k}(\epsilon)\int_{\Omega}\eta\varphi_{\alpha}^{2}e^{2\alpha m}dx\right)\\
&\leq \epsilon\int_{\Omega}|\nabla\varphi_{\alpha}|^{2} e^{2\alpha
  m}dx+C(\epsilon)\int_{\Omega}\eta\varphi_{\alpha}^{2}e^{2\alpha m}dx
\end{aligned}
\end{equation}
for some constant $C(\epsilon)$ which is independent of $\alpha$.

By means of \eqref{ineq5}, we have
\begin{equation}\label{ineq-eng}
  \begin{aligned}
\liminf_{\alpha\to\infty}\lambda_1(\alpha)&=\lim_{\alpha\to\infty}\int_{\Omega}(e^{2\alpha
  m}|\nabla\varphi_{\alpha}|^{2}+Ve^{2\alpha
  m}\varphi_{\alpha}^{2})dx+\int_{\partial\Omega}\beta e^{2\alpha
  m}\varphi_{\alpha}^{2}d\mathcal{H}^{N-1}\\
&\geq\lim_{\alpha\to\infty}\int_{\Omega}(e^{2\alpha
  m}|\nabla\varphi_{\alpha}|^{2}+Ve^{2\alpha
  m}\varphi_{\alpha}^{2})dx+\int_{\partial\Omega^{-}}\beta e^{2\alpha
  m}\varphi_{\alpha}^{2}d\mathcal{H}^{N-1}\\
&\geq \lim_{\alpha\to\infty}\left(1-\beta^{*}\epsilon\right)\int_{\Omega}|\nabla\varphi_{\alpha}|^{2}e^{2\alpha
  m}dx+\int_{\Omega}V e^{2\alpha  m}\varphi_{\alpha}^{2}dx-\beta^{*}C(\epsilon)\int_{\Omega} \eta e^{2\alpha
  m}\varphi_{\alpha}^{2}dx,
\end{aligned}
\end{equation}
where $\beta^{*}=\sup_{x\in\partial\Omega^{-}}\beta^{-}(x)$.
We choose $\epsilon$ sufficiently small such that
$1-\beta^{*}\epsilon\geq 0$. Note that
$$
\mathrm{supp}(\eta)\cap\mathrm{supp}
  (\mathrm{\mu})\subset\mathrm{supp}({\eta})\cap(\Sigma_{1}\cup\Sigma_{2})=\emptyset
$$
due to \eqref{set} and the choice of $\eta$. It then follows from the weak convergence of $\mu_{\alpha}$ (see \eqref{weakcon})  that
\[
\lim_{\alpha\to\infty}\int_{\Omega}\eta\varphi_{\alpha}^{2}e^{2\alpha m}dx=
\lim_{\alpha\to\infty}\int_{\mathbb{R}^{N}}\eta
d\mu_{\alpha}=\int_{\mathbb{R}^{N}}\eta d \mu=0.
\]
Therefore, because of \eqref{ineq-eng}, we can deduce
\[
\liminf_{\alpha\to\infty}\lambda_1(\alpha)
\geq\lim_{\alpha\to\infty}\int_{\Omega}V\varphi_{\alpha}^{2}e^{2\alpha m}dx\geq
\min_{x\in \Sigma_{1}\cup\Sigma_{2}}V(x).
  \]
This verifies the assertion (ii).

We finally prove (iii) under the conditions (A1), (A2) and
$\Sigma_{1}\cup\Sigma_{2}\cup\Sigma_{3}=\emptyset$.
Arguing indirectly, we suppose that there exists a sequence of $\alpha$, still denoted by itself, such that
\[
\lim_{\alpha\to\infty}\lambda_1(\alpha)<\infty,
\]
which implies \eqref{ineq4} holds. Then, the same analysis as in the proof of the assertion (ii) yields $\mu(\Omega_{i})=0$ for $i=1,\cdots,6$ and $\mu(\Sigma_{4})=0$. Thus, $\text{supp}(\mu)\subset\Sigma_{1}\cup\Sigma_{2}$, which, in turn, gives $\mu(\overline{\Omega})=0$
thanks to $\Sigma_{2}\cup\Sigma_{3}=\emptyset$. This contradicts with $\mu(\overline{\Omega})=1$.
Consequently, we must have $\lim_{\alpha\to\infty}\lambda_1(\alpha)=\infty$. The proof is now completed.
\end{proof}
\subsection{Proof of Theorem \ref{degenerate}}
This subsection is devoted to the proof of Theorem \ref{degenerate}.

\begin{proof}[Proof of Theorem \ref{degenerate}]
Theorem \ref{degenerate}(i) follows directly from Lemma \ref{unbounded}.

Denote $\Omega_{i}\ (i=1,3)$ as before, and let $\{\alpha_{k}\}_{k\in\mathbb{N}}$ be a sequence such that
$\lim_{k\to\infty}\alpha_{k}=\infty$ and $\mu_{\alpha_{k}}$ weakly
converges to a Radon measure $\mu$ with $\mathrm{supp}(\mu)\subset
\overline{\Omega}$ and $\mu(\overline{\Omega})=1$. It follows from the
same argument as in \cite{CL2008} that $\mu(\Omega_{i})=0$ for
$i=1,3$.

Define
$$
\tilde{\Omega}_{2}=\{x\in \Omega\setminus\Sigma:\ \nabla m(x)=0\},
$$
$$
\tilde{\Omega}_{4}=\{x\in\partial\Omega\setminus\Sigma:\ \nabla m(x)=0\ \text{ or }\
|\nabla m(x)|=-\frac{\partial}{\partial n} m(x)>0\}
$$
and
$$
 \Sigma_{4}=\{x\in\Sigma\cap\partial\Omega:\ \ \beta(x)>0\}.
 $$

If $\Sigma_{3}=\emptyset$ and $\Sigma_{1}\cup\Sigma_{2}\neq\emptyset$,
similarly to the proof of Theorem \ref{general}, we have
 \begin{equation}\label{ineq-z6}
\int_{\Omega}|\nabla w_{\alpha_{k}}-\alpha w_{\alpha_{k}}\nabla m|^{2}dx<C
 \end{equation}
for some constant $C$, independent of $k$. Then Lemma
\ref{lemma-degenerate} below tells us that $\mu(\tilde{\Omega}_{2})=0$
and $\mu(\tilde{\Omega}_{4})=0$, which implies
$\mathrm{supp}(\mu)\subset \cup_{i=1}^{4}\Sigma_{i}$. Now, the same
arguments as in Theorem  \ref{nonnegative} and Theorem \ref{general} can be used to deduce
Theorem \ref{degenerate}(ii) and (iii).
\end{proof}

\begin{lemma}\label{lemma-degenerate}
  Let $x_{0}\in{\overline{\Omega}}\setminus\Sigma$ be a critical point of $m$. Assume that {\rm (A3)} and \eqref{ineq-z6} hold. Then
we have $\mu(\{x_{0}\})=0$.
\end{lemma}

\begin{proof} \textbf{Case 1}. $x_0\in\Omega$. We assume that (a)-(i) of (A3) holds; the proof for case (b) is similar. Let $\Gamma,\,\nu,\,\xi,\,r_{0}, \,\delta,\,D_1,\,D_2$ be given as in (a)-(i) of (A3).
For $x\in B(x_0,r_0)$, we define the signed distance function
\begin{equation*}
\tilde d(x)=
\begin{cases}
\inf_{y\in\Gamma}|x-y|, &\text{ if }x\in D_1,\\
-\inf_{y\in\Gamma}|x-y|, &\text{ if }x\in D_2.
\end{cases}
\end{equation*}
We may assume that $r_0$ is sufficiently small such that $\tilde d(x)$ is $C^{1}$ and
$\xi(x)\cdot\nu(x)\geq\delta/2$ for all $x\in B(x_0,r_0)$, where $\tilde\nu(x)=\nabla\tilde d(x)$ and $\tilde\nu(x)=\nu(x)$ for $x\in\Gamma\cap B(x_0,r_0)$.
Let $\eta(x,t)$ be the solution of
\[
\frac{d}{dt}\eta(x,t)=\frac{\xi(\eta(x,t))}{\xi(\eta(x,t))\cdot\tilde \nu(\eta(x,t))}, \ \ \eta(x,0)=x.
\]
Then $d(\eta(x,t))=t$ for all $x\in\Gamma$ and $t\in(-r_0,r_0)$.
Given $r_1\in(0,r_0)$, let us denote
\[
\Gamma_t=\{\eta(x,t):\ \ x\in\Gamma\cap B(x_0,r_1)\}.
\]
We may further choose $r_1$ sufficiently small such that $\Gamma_t\subset B(x_0,r_0)$ for all $t\in(-r_1,r_1)$. Then for each $t\in(-r_1,r_1)$,
\[
\Phi^t:\ \ \Gamma_0\to\Gamma_t,\ \ \ \Phi^t(x)=\eta(x,t)
\]
is a differomorphism. Hence, there exists a positive constant $\sigma$, which
is independent of $t$, such that
\[
\sigma\mathcal{H}^{N-1}|_{\Gamma_t}\leq \Phi^{t}_*(\mathcal{H}^{N-1}|_{\Gamma_0})\leq \sigma^{-1}\mathcal{H}^{N-1}|_{\Gamma_t}.
\]

Suppose towards a contradiction that $\mu(\{x_0\})>0$. Then, for any given $\epsilon>0$,
  \[
\int_{-\epsilon r_1}^{\epsilon
  r_1}\int_{\Gamma_{t}}w_{\alpha_{k}}^{2}d\mathcal{H}^{N-1}dt=\int_{\cup_{t\in(-\epsilon r_1,\epsilon r_1)}\Gamma_{t}}w_{\alpha_{k}}^{2}dx\geq \frac{\mu(\{x_0\})}{2}
\]
 for all large $k$.
Thus, there exist constants $s_{k}\in (-\epsilon r_1,\epsilon r_1)$ such that
 \begin{equation}\label{ineq-z7}
\int_{\Gamma_{s_{k}}}w_{\alpha_{k}}^{2}d\mathcal{H}^{N-1}\geq
\frac{\mu(\{x_0\})}{4\epsilon r_1},\ \ \mbox{for all large}\ k.
   \end{equation}
Without loss of generality, we may assume that $s_{k}\in
[0,\epsilon r_1)$.

On the other hand, it follows from the same argument as in \cite[Lemma 3.1]{CL2008} that $\mu(\Omega_{1})=0$.
Since all critical points are isolated, we may assume that
$\inf_{\cup_{t\in(r_1-\epsilon r_1,r_1)}\Gamma_t}|\nabla m(x)|>0.$
This, together with the definitions of $\Omega_1$ and the weak convergence of $w_{\alpha_k}^{2}\to\mu$, yields that
\[
\lim_{k\to\infty}\int_{\cup_{t\in(r_1-\epsilon r_1,r_1)}\Gamma_t}w_{\alpha_k}^{2}dx\leq\mu(\Omega_{1})=0,
  \]
from which it follows, for all large $k$, that
  \[
\int_{r_1-\epsilon r_1}^{r_1}\int_{\Gamma_{t}}w_{\alpha_{k}}^{2}d\mathcal{H}^{N-1}dt=\int_{\cup_{t\in(r_1-\epsilon r_1,r_1)}\Gamma_t}w_{\alpha_{k}}^{2}dx<\epsilon\mu(\{x_0\}).
\]
Thus, we have
 \begin{equation}\label{ineq-z8}
\int_{\Gamma_{t_{k}}}w_{\alpha_{k}}^{2}d\mathcal{H}^{N-1}<\frac{\mu(\{x_0\})}{r_1}
   \end{equation}
for some constant $t_{k}\in (r_1-\epsilon r_1,r_1)$ and all large $k$.

By the assumption, we have, for all $t\in (0,r_1)$,
$\xi(x)\cdot\nabla m(x)\geq 0$ on $\Gamma_t$. As a consequence, by means of \eqref{ineq-z6}, \eqref{ineq-z7},  and \eqref{ineq-z8}, we obtain
\begin{equation*}
\begin{aligned}
\frac{\sigma\mu(\{x_0\})}{4\epsilon r_1}-\frac{\mu(\{x_0\})}{\sigma r_1}&\leq \sigma\int_{\Gamma_{s_k}}w_{\alpha_k}^2d\mathcal{H}^{N-1}-\sigma^{-1}\int_{\Gamma_{t_k}}w_{\alpha_k}^2d\mathcal{H}^{N-1}\\
&\leq \int_{\Gamma_{0}}w_{\alpha_k}^2(\eta(x,s_k))d\mathcal{H}^{N-1}-\int_{\Gamma_{0}}w_{\alpha_k}^2(\eta(x,t_k))d\mathcal{H}^{N-1}\\
&=\int_{\Gamma_0}\int_{s_k}^{t_k}-\partial_t w_{\alpha_k}^2(\eta(x,t))dtd\mathcal{H}^{N-1}\\
&=\int_{\Gamma_0}\int_{s_k}^{t_k}-2w_{\alpha_k}(\eta(x,t))\nabla w_{\alpha_k}(\eta(x,t))\cdot\frac{\xi(\eta(x,t))}{\xi(\eta(x,t))\cdot\tilde\nu(\eta(x,t))}dtd\mathcal{H}^{N-1}\\
&\leq\frac{2}{\delta} \int_{\Gamma_0}\int_{s_k}^{t_k}2\left(\nabla w_{\alpha_k}(\eta(x,t))\cdot\xi(\eta(x,t))\right)^{-}w_{\alpha_k}(\eta(x,t))dtd\mathcal{H}^{N-1}\\
&\leq \frac{2}{\delta}\int_{\Gamma_0}\int_{s_k}^{t_k}\left[\left(\nabla w_{\alpha_k}(\eta(x,t))\cdot\xi(\eta(x,t))\right)^{-}\right]^2+w_{\alpha_k}^2(\eta(x,t))dtd\mathcal{H}^{N-1}\\
&=\frac{2}{\delta}\int_{s_k}^{t_k} \int_{\Gamma_0}\left[\left(\nabla w_{\alpha_k}(\eta(x,t))\cdot\xi(\eta(x,t))\right)^{-}\right]^2+w_{\alpha_k}^2(\eta(x,t))d\mathcal{H}^{N-1}dt\\
&\leq\frac{2}{\delta\sigma}\int_{s_k}^{t_k} \int_{\Gamma_t}\left[\left(\nabla w_{\alpha_k}(x)\cdot\xi(x)\right)^{-}\right]^2+w_{\alpha_k}^2(x)d\mathcal{H}^{N-1}dt\\
&\leq \frac{2}{\delta\sigma}\int_{\cup_{t\in[s_k,t_k]}\Gamma_t}\left[\left(\nabla w_{\alpha_k}(x)\cdot\xi(x)\right)^{-}\right]^2+w_{\alpha_k}^2(x)dx\\
&\leq \frac{2}{\delta\sigma}\int_{\left(\cup_{t\in[s_k,t_k]}\Gamma_t\right)\cap\{x:\ \xi(x)\cdot\nabla m(x)\geq 0\}}\left[\nabla w_{\alpha_k}\cdot\xi-\alpha_k w_{\alpha_k}\xi\cdot\nabla m\right]^2dx+\int_{\Omega}w_{\alpha_k}^2dx\\
&\leq\frac{2}{\delta\sigma}\int_{\Omega}\left[\nabla w_{\alpha_k}\cdot\xi-\alpha_k w_{\alpha_k}\xi\cdot\nabla m\right]^2+w_{\alpha_k}^2dx\\
&\leq\frac{2}{\delta\sigma}\int_{\Omega}\left|\nabla w_{\alpha_k}-\alpha_k w_{\alpha_k}\nabla m\right|^2+w_{\alpha_k}^2dx\\
&\leq \frac{2(C+1)}{\delta\sigma},
\end{aligned}
\end{equation*}
a contradiction provided $\epsilon$ is small enough, and in turn $\mu(\{x_{0}\})=0$ holds.

\textbf{Case 2}. $x_0\in\partial\Omega$. We assume that (c)-(ii) in (A3) holds; the other cases can be handled in a similar manner. Let $\Gamma,\,\nu,\,\xi,\,r_{0}, \,\delta,\,D_1,\,D_2$ be given as in (c)-(ii) of (A3).
Also define the signed distance function
$\tilde d(x)$, $\tilde\nu(x)$ and $\eta(x,t)$ as in {\bf Case 1} such that $\xi(x)\cdot\nu(x)\geq\delta/2$ for all $x\in B(x_0,r_0)$. Thus $\tilde d(\eta(x,t))=t$ for all $x\in\Gamma$ and $t\in(-r_0,r_0)$.
For $r_1\in(0,r_0)$, we also set $\Gamma_t=\{\eta(x,t):\ \ x\in\Gamma\cap B(x_0,r_1)\}$.
We may choose $r_1$ to be so small that $\Gamma_t\subset B(x_0,r_0)$ for all $t\in(-r_1,r_1)$.
Given $s,t\in(-r_1,r_1)$,  then
$\Phi^{s,t}:\ \Gamma_s\to\Gamma_t,\ \,\Phi^{s,t}(x)=\eta(x,t-s)$
is a differomorphism, and there exists a constant $\sigma>0$, independent of $s$ and $t$, such that
\[
\sigma\mathcal{H}^{N-1}|_{\Gamma_t}\leq \Phi^{s,t}_*(\mathcal{H}^{N-1}|_{\Gamma_s})\leq \sigma^{-1}\mathcal{H}^{N-1}|_{\Gamma_t}.
\]

Let $\tilde{\Gamma}_t=\Gamma_t\cap\Omega$. Since
$\xi(x)\cdot n(x)\leq0$ on $\partial\Omega\cap \overline{D}_1$ and
$\xi(x)\cdot n(x)\geq 0$ on $\partial\Omega\cap \overline{D}_2$,
we have
\[
\Phi^{s,t}(\tilde{\Gamma}_s)\subset\tilde{\Gamma}_t,\ \ \forall 0\leq s\leq t\leq r_1\ \ \mbox{and}\ \
\Phi^{s,t}(\tilde{\Gamma}_s)\subset\tilde{\Gamma}_t,\ \ \forall 0\geq s\geq t\geq -r_1.
\]

As before, we suppose that $\mu(\{x_0\})>0$. Then, for any given $\epsilon>0$,
  \[
\int_{-\epsilon r_1}^{\epsilon
  r_1}\int_{\tilde{\Gamma}_{t}}w_{\alpha_{k}}^{2}d\mathcal{H}^{N-1}dt=\int_{\cup_{t\in(-\epsilon r_1,\epsilon r_1)}\tilde{\Gamma}_{t}}w_{\alpha_{k}}^{2}dx\geq \frac{\mu(\{x_0\})}{2}
\]
 for all large $k$, which gives rise to
 \begin{equation}\label{ineq-z10}
\int_{\tilde{\Gamma}_{s_{k}}}w_{\alpha_{k}}^{2}d\mathcal{H}^{N-1}\geq
\frac{\mu(\{x_0\})}{4\epsilon r_1}
   \end{equation}
for $s_{k}\in (-\epsilon r_1,\epsilon r_1)$ for all large $k$.
We may assume that $s_{k}\in [0,\epsilon r_1)$.

By a similar reasoning as to derive \eqref{ineq-z8}, for all large $k$, it follows that
 \begin{equation}\label{ineq-z11-a}
\int_{\tilde{\Gamma}_{t_{k}}}w_{\alpha_{k}}^{2}dx<\frac{\mu(\{x_0\})}{r_1}
   \end{equation}
for some constant $t_{k}\in (r_1-\epsilon r_1,r_1)$ and all large $k$.
In addition, we may assume that for all $t\in (0,r_1)$,
$\xi(x)\cdot\nabla m(x)\geq 0$ on $\Gamma_t$.

Now, using \eqref{ineq-z6}, \eqref{ineq-z10},  \eqref{ineq-z11-a} and the fact that
$\Phi^{t_k,s_k}(\tilde{\Gamma}_{t_k})\subset\tilde{\Gamma}_{s_k}$, \
similarly to {\bf Case 1}, we have
\begin{equation*}
\begin{aligned}
\frac{\mu(\{x_0\})}{4\epsilon r_1}-\frac{\mu(\{x_0\})}{\sigma r_1}&\leq \int_{\tilde{\Gamma}_{s_k}}w_{\alpha_k}^2d\mathcal{H}^{N-1}-\sigma^{-1}\int_{\tilde{\Gamma}_{t_k}}w_{\alpha_k}^2d\mathcal{H}^{N-1}\\
&\leq \int_{\tilde{\Gamma}_{s_k}}w_{\alpha_k}^2d\mathcal{H}^{N-1}-\sigma^{-1}\int_{\Phi^{t_k,s_k}(\tilde{\Gamma}_{t_k})}w_{\alpha_k}^2d\mathcal{H}^{N-1}\\
&\leq \int_{\tilde{\Gamma}_{s_k}}w_{\alpha_k}^2(x)d\mathcal{H}^{N-1}-\int_{\tilde{\Gamma}_{s_k}}w_{\alpha_k}^2(\eta(x,t_k-s_k))d\mathcal{H}^{N-1}\\
&=\int_{\tilde{\Gamma}_{s_k}}\int_{0}^{t_k-s_k}-\partial_t w_{\alpha_k}^2(\eta(x,t))dtd\mathcal{H}^{N-1}\\
&=\int_{\tilde{\Gamma}_{s_k}}\int_{0}^{t_k-s_k}-2w_{\alpha_k}(\eta(x,t))\nabla w_{\alpha_k}(\eta(x,t))\cdot\frac{\xi(\eta(x,t))}{\xi(\eta(x,t))\cdot\tilde\nu(\eta(x,t))}dtd\mathcal{H}^{N-1}\\
&\leq\frac{2}{\delta} \int_{\tilde{\Gamma}_{s_k}}\int_{0}^{t_k-s_k}2\left(\nabla w_{\alpha_k}(\eta(x,t))\cdot\xi(\eta(x,t))\right)^{-}w_{\alpha_k}(\eta(x,t))dtd\mathcal{H}^{N-1}\\
&\leq \frac{2}{\delta}\int_{\tilde{\Gamma}_{s_k}}\int_{0}^{t_k-s_k}\left[\left(\nabla w_{\alpha_k}(\eta(x,t))\cdot\xi(\eta(x,t))\right)^{-}\right]^2+w_{\alpha_k}^2(\eta(x,t))dtd\mathcal{H}^{N-1}\\
&=\frac{2}{\delta}\int_{0}^{t_k-s_k}\int_{\tilde{\Gamma}_{s_k}}\left[\left(\nabla w_{\alpha_k}(\eta(x,t))\cdot\xi(\eta(x,t))\right)^{-}\right]^2+w_{\alpha_k}^2(\eta(x,t))d\mathcal{H}^{N-1}dt\\
&\leq\frac{2}{\delta\sigma}\int_{s_k}^{t_k} \int_{\tilde{\Gamma}_t}\left[\left(\nabla w_{\alpha_k}(x)\cdot\xi(x)\right)^{-}\right]^2+w_{\alpha_k}^2(x)d\mathcal{H}^{N-1}dt\\
&\leq \frac{2}{\delta\sigma}\int_{\left(\cup_{t\in[s_k,t_k]}\tilde{\Gamma}_t\right)\cap\{x:\ \xi(x)\cdot\nabla m(x)\geq 0\}}\left[\nabla w_{\alpha_k}\cdot\xi-\alpha_k w_{\alpha_k}\xi\cdot\nabla m\right]^2dx+\int_{\Omega}w_{\alpha_k}^2dx\\
&\leq\frac{2}{\delta\sigma}\int_{\Omega}\left[\nabla w_{\alpha_k}\cdot\xi-\alpha_k w_{\alpha_k}\xi\cdot\nabla m\right]^2+w_{\alpha_k}^2dx\\
&\leq \frac{2(C+1)}{\delta\sigma},
\end{aligned}
\end{equation*}
which is impossible if we send $\epsilon$ to be small enough. Therefore, we have
proved $\mu(\{x_{0}\})=0$.
\end{proof}

\section{One dimensional problem}

In this section, we consider the eigenvalue problem with a general boundary condition in one space dimension:
\begin{equation}
 \left\{\begin{array}{ll}
 \medskip
 \displaystyle
 -\varphi''(x)-2\alpha m'(x)\varphi'(x)+V(x)\varphi(x)=\lambda\varphi(x),\ &0<x<1,\\
 \displaystyle
 -\hbar_1\varphi'(0)+\ell_1\varphi(0)=\hbar_2\varphi'(1)+\ell_2\varphi(1)=0,
 \end{array}
 \right.
 \label{p}
\end{equation}
where $m\in C^2([0,1]),\,V\in C([0,1])$, and the constants
$\hbar_i,\,\ell_i\,(i=1,2)$ satisfy $|\hbar_i|+|\ell_i|>0\,(i=1,2)$.

As in \cite{PZ2018}, in this section, unless otherwise specified, we always assume that
\vskip4pt
\noindent\textbf{(A4)}:\ \ $m$ is not constant and $m'(x)$\ changes sign at most finitely many times on $[0,1]$.
\vskip3pt
\noindent Hence, $m$ can have various natural kinds of degeneracy.

\vskip4pt
Before stating the main results of this section, we need to classify the set of points of  local maximum points of $m$ and introduce some notations as in \cite{PZ2018}.

\begin{definition}\label{def-2}

\vskip5pt \item {\rm (i)} A {\bf point of local maximum} of $m$ is a point $x\in[0,1]$ such that there is a small $\epsilon_0>0$ such that $m(x)\geq m(y)$ in $(x-\epsilon_0,x+\epsilon_0)\cap[0,1]$, and such $x$ is said to be an {\bf interior point of local maximum} if $x\in(0,1)$;

\vskip5pt \item {\rm (ii)} An {\bf isolated point of local maximum} of $m$ is a point $x^{I}\in[0,1]$ such that there is a small $\epsilon_0>0$ such that $m(x^{I})>m(x)$ in $(x^{I}-\epsilon_0,x^{I}+\epsilon_0)\cap[0,1]\setminus\{x^{I}\}$;

\vskip5pt \item {\rm (iii)} A {\bf segment of local maximum} of $m$ is a closed interval $[a,b]\subset[0,1]$ such that there is a small $\epsilon_0>0$ such that
$m$ is constant on the closed interval $[a,b]$, $m$ is monotone (non-increasing or non-decreasing) on $[a-\epsilon_0,a]\cap[0,1]$ and $[b,b+\epsilon_0]\cap[0,1]$, and
for any small $\epsilon>0$, $\{x\in[0,1]:\ m'(x)\not=0\}\cap(a-\epsilon,a)\not=\emptyset$ if $0<a$ and $\{x\in[0,1]:\ m'(x)\not=0\}\cap(b,b+\epsilon)\not=\emptyset$ if $b<1$.

\end{definition}

Note that the definition of an isolated point of local maximum in Definition \ref{def-2} is exactly that in  Definition \ref{def-1}(iii) under the assumption (A4). Moreover, if (A4) holds,  clearly $m$ admits at most finitely many isolated points of local maximum. For later purpose, we will use the same notations
$$
\mbox{$[a^i,b^j]\subset[0,1]$\ \ \ with $0\leq a<b\leq 1$,\ $i,j\in\{I,\, D\}$}
$$
as in \cite{PZ2018} to distinguish all possible segments of local maximum of $m$; see (1)-(5) in page 2527 of \cite{PZ2018}. Here, the capital letters $I$ and $D$ represent increasing and decreasing, respectively. For example, $[a^{{I}},b^{{D}}]$ means that $m$ increases in a left neighbourhood of $a^{{I}}$ and decreases in a right neighbourhood of $b^{{D}}$ while $m$ is constant on $[a^{{I}},b^{{D}}]$. Under the assumption (A4), $m$ admits at most finitely many segments $[a^{{I}},b^{{D}}]$ and $[a^{{D}},b^{{I}}]$ of local maximum, and it has at most countably many segments $[a^{{I}},b^{{I}}]$ and $[a^{{D}},b^{{D}}]$ of local maximum.

As in \cite{PZ2018}, we have to introduce some further notations as follows. Given a closed interval $[a,b]\subset[0,1]$ with $0<a<b<1$ and $i,j\in\{\mathcal{N},\, \mathcal{D}\}$, we denote by $\lambda_1^{{ij}}(a,b)$
the principal eigenvalue of the elliptic eigenvalue problem:
 \begin{equation}
 -\varphi''(x)+V(x)\varphi(x)=\lambda\varphi(x),\ a<x<b
 \nonumber
 \end{equation}
with the Neumann boundary condition (Dirichlet boundary condition, respectively) at the left boundary point $a$ if $i=\mathcal{N}$ (if $i=\mathcal{D}$, respectively), and Neumann boundary condition (Dirichlet boundary condition, respectively) at the right boundary point $b$ if $j=\mathcal{N}$ (if $j=\mathcal{D}$, respectively).

Given $[0,b]\subset[0,1]$ with $0<b<1$ and $i\in\{\mathcal{N},\, \mathcal{D}\}$, denote by $\lambda_1^{{\mathcal{R}i}}(0,b)$
the principal eigenvalue of
 \begin{equation}
 -\varphi''(x)+V(x)\varphi(x)=\lambda\varphi(x),\ 0<x<b;\ \ \ -\hbar_1\varphi'(0)+\ell_1\varphi(0)=0
 \nonumber
 \end{equation}
with the Neumann boundary condition (Dirichlet boundary condition, respectively ) at $b$ if $i=\mathcal{N}$ (if $i=\mathcal{D}$, respectively).

Given $[a,1]\subset[0,1]$ with $0<a<1$ and $i\in\{\mathcal{N},\, \mathcal{D}\}$, denote by $\lambda_1^{{i\mathcal{R}}}(a,1)$
the principal eigenvalue of
 \begin{equation}
 -\varphi''(x)+V(x)\varphi(x)=\lambda\varphi(x),\ a<x<1; \ \ \ \hbar_2\varphi'(1)+\ell_2\varphi(1)=0
 \nonumber
 \end{equation}
with the Neumann boundary condition (Dirichlet boundary condition, respectively ) at $a$ if $i=\mathcal{N}$ (if $i=\mathcal{D}$, respectively).

In view of (A4), the set $\mathcal{M}$ of local maxima of $m$ consists of
 \begin{equation}
 \mathcal{M}=\cup_{i=1}^9\mathcal{M}_i,
 \nonumber
 \end{equation}
where
 \begin{equation}
 \left.\begin{array}{lll}
 \medskip
 \displaystyle
 \mathcal{M}_1=\{x^{I}_i\}_{i=1}^{h_1},\ \ \ \ \mathcal{M}_2=\{[a_i^{{I}},b_i^{{I}}]\}_{i=1}^{h_2},
 \ \ \ \ \mathcal{M}_3=\{[a_i^{{I}},b_i^{{D}}]\}_{i=1}^{h_3},\ \ \ \ \mathcal{M}_4=\{[a_i^{{D}},b_i^{{I}}]\}_{i=1}^{h_4},\\
 \medskip
 \displaystyle
 \mathcal{M}_5=\{[a_i^{{D}},b_i^{{D}}]\}_{i=1}^{h_5},\ \ \mathcal{M}_6\subset\{[0,a^{{I}}]\},\ \
 \mathcal{M}_7\subset\{[0,a^{{D}}]\},\ \ \mathcal{M}_8\subset\{[a^{{I}},1]\},\ \ \mathcal{M}_9\subset\{[a^{{D}},1]\}.
 \end{array}
 \right.
 \nonumber
 \end{equation}
Note that $h_1,\,h_3,\,h_4$ are finite integers while $h_2,\,h_5$ may be finite integers or infinity. Though some $\mathcal{M}_i$ may be empty, obviously $\mathcal{M}\not=\emptyset$.

For simplicity, let us also set
 \begin{equation}
 \left.\begin{array}{lll}
 \medskip
 \displaystyle
 \mathfrak{L}&=&
 \displaystyle
 \min\Big\{\min\{\lambda_1^{\mathcal{ND}}(a^I_i,b^I_i):\ [a^I_i,b^I_i]\in\mathcal{M}_2\},\ \
 \min\{\lambda_1^{\mathcal{NN}}(a^I_i,b^D_i):\ [a^I_i,b^D_i]\in\mathcal{M}_3\},\\
 \medskip
 &&\ \ \ \ \ \ \ \
 \displaystyle
 \min\{\lambda_1^{\mathcal{DD}}(a^D_i,b^I_i):\ [a^D_i,b^I_i]\in\mathcal{M}_4\},\ \
 \min\{\lambda_1^{\mathcal{DN}}(a^D_i,b^D_i):\ [a^D_i,b^D_i]\in\mathcal{M}_5\}\Big\}.
 \end{array}
 \right.
 \nonumber
 \end{equation}

In \cite{PZ2018}, the situation
 \begin{equation}\label{regular}
 \hbar_1\ell_1\geq0,\ \ \hbar_2\ell_2\geq0
  \end{equation}
was studied. In this case, it is clear that $\lambda_1(\alpha)\geq\min_{x\in[0,1]}V(x)$ for all $\alpha$.
The main result for problem \eqref{p} obtained in \cite{PZ2018} reads as follows.

\begin{thm}\label{theorem-R} Assume that {\rm (A4)} and \eqref{regular} hold. Then we have the following assertions.

 \begin{itemize}

\item[(i)] $\lim\limits_{\alpha\to\infty}\lambda_1(\alpha)=\infty$ if and only if

{\rm{(i-1)}}\ $\mathcal{M}=\mathcal{M}_1\subset\{0,1\}$ when $\ell_1\not=0$ and $\ell_2\not=0$.\

{\rm{(i-2)}}\ $\mathcal{M}=\mathcal{M}_1=\{0\}$ when $\ell_1\not=0$ and $\ell_2=0$.\

{\rm{(i-3)}}\ $\mathcal{M}=\mathcal{M}_1=\{1\}$ when $\ell_1=0$ and $\ell_2\not=0$.

\item[(ii)] Otherwise, $-\infty<\lim\limits_{\alpha\to\infty}\lambda_1(\alpha)<\infty$, and
 \begin{equation}
 \left.\begin{array}{lll}
 \medskip
 \displaystyle
 \lim\limits_{\alpha\to\infty}\lambda_1(\alpha)&=&
 \medskip
 \displaystyle
 \min\Big\{\mathfrak{L},\ \ \min\{V(x):\ x\in\mathcal{M}_1\setminus\{0,\,1\}\},\ \ \lambda_1^{\mathcal{RD}}(0,a^I)\ (\mbox{if}\ [0,a^I]\in\mathcal{M}_6),\\
 &&\ \ \ \ \ \ \
 \medskip
 \displaystyle
 \lambda_1^{\mathcal{RN}}(0,a^D) \ (\mbox{if}\ [0,a^D]\in\mathcal{M}_7),\ \ \lambda_1^{\mathcal{NR}}(a^I,1)\ (\mbox{if}\ [a^I,1]\in\mathcal{M}_8),\\
 &&\ \ \ \ \ \ \
 \medskip
 \displaystyle
 \lambda_1^{\mathcal{DR}}(a^D,1)\ (\mbox{if}\ [a^D,1]\in\mathcal{M}_9),\ \ V(0)\ (\mbox{if\ $0\in\mathcal{M}_1$ and $\ell_1=0$}),\\
 &&\ \ \ \ \ \ \
 \displaystyle
 V(1)\ (\mbox{if\ $1\in\mathcal{M}_1$ and $\ell_2=0$})\, \Big\}.
 \end{array}
 \right.
 \nonumber
 \end{equation}
\end{itemize}
\end{thm}

In the sequel, we will consider all the remaining cases. Our first result concerns the case of
\begin{equation}\label{nonreg-1}
 \hbar_1\ell_1<0,\ \ \ \hbar_2\ell_2\geq0.
  \end{equation}
Indeed, we have the following conclusion.

\begin{theorem}\label{th1.2-n1} Assume that {\rm (A4)} and \eqref{nonreg-1} hold. Then we have the following assertions.

 \begin{itemize}

\item[(i)] $\lim\limits_{\alpha\to\infty}\lambda_1(\alpha)=\infty$ if and only if
$\mathcal{M}=\mathcal{M}_1=\{1\}$ when $\ell_2\not=0$.

\item[(ii)] $\lim\limits_{\alpha\to\infty}\lambda_1(\alpha)=-\infty$ if and only if
$\mathcal{M}=\mathcal{M}_1=\{0\}$.

\item[(iii)] Otherwise, $-\infty<\lim\limits_{\alpha\to\infty}\lambda_1(\alpha)<\infty$ and
 \begin{equation}
 \left.\begin{array}{lll}
 \medskip
 \displaystyle
 \lim\limits_{\alpha\to\infty}\lambda_1(\alpha)&=&
 \medskip
 \displaystyle
 \min\Big\{\mathfrak{L},\ \ \min\{V(x):\ x\in\mathcal{M}_1\setminus\{0,\,1\}\},\ \ \lambda_1^{\mathcal{RD}}(0,a^I)\ (\mbox{if}\ [0,a^I]\in\mathcal{M}_6),\\
 &&\ \ \ \ \ \ \
 \medskip
 \displaystyle
 \lambda_1^{\mathcal{RN}}(0,a^D) \ (\mbox{if}\ [0,a^D]\in\mathcal{M}_7),\ \ \lambda_1^{\mathcal{NR}}(a^I,1)\ (\mbox{if}\ [a^I,1]\in\mathcal{M}_8),\\
 &&\ \ \ \ \ \ \
 \medskip
 \displaystyle
 \lambda_1^{\mathcal{DR}}(a^D,1)\ (\mbox{if}\ [a^D,1]\in\mathcal{M}_9),\ V(1)\ (\mbox{if\ $1\in\mathcal{M}_1$ and $\ell_2=0$})\, \Big\}.
 \end{array}
 \right.
 \nonumber
 \end{equation}

\end{itemize}
\end{theorem}

Parallel to case \eqref{nonreg-1}, for the case
\begin{equation}\label{nonreg-2}
 \hbar_1\ell_1\geq0,\ \ \ \hbar_2\ell_2<0,
  \end{equation}
we can state the following result.

\begin{theorem}\label{th1.2-n2} Assume that {\rm (A4)} and \eqref{nonreg-2} hold. Then we have the following assertions.

 \begin{itemize}

\item[(i)] $\lim\limits_{\alpha\to\infty}\lambda_1(\alpha)=\infty$ if and only if
$\mathcal{M}=\mathcal{M}_1=\{0\}$ when $\ell_1\not=0$.

\item[(ii)] $\lim\limits_{\alpha\to\infty}\lambda_1(\alpha)=-\infty$ if and only if
$\mathcal{M}=\mathcal{M}_1=\{1\}$.

\item[(iii)] Otherwise, $-\infty<\lim\limits_{\alpha\to\infty}\lambda_1(\alpha)<\infty$ and
 \begin{equation}
 \left.\begin{array}{lll}
 \medskip
 \displaystyle
 \lim\limits_{\alpha\to\infty}\lambda_1(\alpha)&=&
 \medskip
 \displaystyle
 \min\Big\{\mathfrak{L},\ \ \min\{V(x):\ x\in\mathcal{M}_1\setminus\{0,\,1\}\},\ \ \lambda_1^{\mathcal{RD}}(0,a^I)\ (\mbox{if}\ [0,a^I]\in\mathcal{M}_6),\\
 &&\ \ \ \ \ \ \
 \medskip
 \displaystyle
 \lambda_1^{\mathcal{RN}}(0,a^D) \ (\mbox{if}\ [0,a^D]\in\mathcal{M}_7),\ \ \lambda_1^{\mathcal{NR}}(a^I,1)\ (\mbox{if}\ [a^I,1]\in\mathcal{M}_8),\\
 &&\ \ \ \ \ \ \
 \medskip
 \displaystyle
 \lambda_1^{\mathcal{DR}}(a^D,1)\ (\mbox{if}\ [a^D,1]\in\mathcal{M}_9),\ V(0)\ (\mbox{if\ $0\in\mathcal{M}_1$ and $\ell_1=0$})\, \Big\}.
 \end{array}
 \right.
 \nonumber
 \end{equation}

\end{itemize}
\end{theorem}

Finally, we consider the case
\begin{equation}\label{nonreg-3}
 \hbar_1\ell_1<0,\ \ \ \hbar_2\ell_2<0,
  \end{equation}
and can obtain the following result.

\begin{theorem}\label{th1.2-n3} Assume that {\rm (A4)} and \eqref{nonreg-3} hold. Then we have the following assertions.

 \begin{itemize}

\item[(i)] $\lim\limits_{\alpha\to\infty}\lambda_1(\alpha)=-\infty$ if and only if
$\mathcal{M}=\mathcal{M}_1\subset\{0,1\}$.

\item[(ii)] Otherwise, $-\infty<\lim\limits_{\alpha\to\infty}\lambda_1(\alpha)<\infty$ and
 \begin{equation}
 \left.\begin{array}{lll}
 \medskip
 \displaystyle
 \lim\limits_{\alpha\to\infty}\lambda_1(\alpha)&=&
 \medskip
 \displaystyle
 \min\Big\{\mathfrak{L},\ \ \min\{V(x):\ x\in\mathcal{M}_1\setminus\{0,\,1\}\},\ \ \lambda_1^{\mathcal{RD}}(0,a^I)\ (\mbox{if}\ [0,a^I]\in\mathcal{M}_6),\\
 &&\ \ \ \ \ \ \
 \medskip
 \displaystyle
 \lambda_1^{\mathcal{RN}}(0,a^D) \ (\mbox{if}\ [0,a^D]\in\mathcal{M}_7),\ \ \lambda_1^{\mathcal{NR}}(a^I,1)\ (\mbox{if}\ [a^I,1]\in\mathcal{M}_8),\\
 &&\ \ \ \ \ \ \
 \medskip
 \displaystyle
 \lambda_1^{\mathcal{DR}}(a^D,1)\ (\mbox{if}\ [a^D,1]\in\mathcal{M}_9)\, \Big\}.
 \end{array}
 \right.
 \nonumber
 \end{equation}

\end{itemize}
\end{theorem}

\begin{remark}\label{r-result} Regarding Theorem \ref{theorem-R} and Theorems \ref{th1.2-n1}-\ref{th1.2-n3}, we would like to make the following comments.

\begin{itemize}

\item[(i)] Once $\lim\limits_{\alpha\to\infty}\lambda_1(\alpha)$ is finite, it can be given by Theorem \ref{theorem-R}(i), Theorem \ref{th1.2-n1}(iii), Theorem \ref{th1.2-n2}(iii) and Theorem \ref{th1.2-n3}(ii), which exhaust all the possibilities of the parameters $\hbar_i,\,\ell_i\, (i=1,2)$. Moreover, we observe that if $0$ or $1$ is an isolated point of local maximum of $m$, $\lim\limits_{\alpha\to\infty}\lambda_1(\alpha)$, when finite, is affected by such a boundary point only if the Neumann boundary condition is prescribed there.

\item[(ii)] The limit $\lim\limits_{\alpha\to\infty}\lambda_1(\alpha)=-\infty$ only if either $\hbar_1\ell_1<0$ or $\hbar_2\ell_2<0$. Precisely speaking, when $\hbar_1\ell_1<0$, $\lim\limits_{\alpha\to\infty}\lambda_1(\alpha)=-\infty$ if and only if $0$ is an isolated local maximum of $m$ (see Theorem \ref{th1.2-n1}(ii)); when $\hbar_2\ell_2<0$, $\lim\limits_{\alpha\to\infty}\lambda_1(\alpha)=-\infty$ if and only if $1$ is an isolated local maximum of $m$ (see Theorem \ref{th1.2-n2}(ii)); when $\hbar_1\ell_1<0$ and $\hbar_2\ell_2<0$, $\lim\limits_{\alpha\to\infty}\lambda_1(\alpha)=-\infty$ if and only if $1$ or $1$ is an isolated local maximum of $m$ (see Theorem \ref{th1.2-n3}(i)). It should be stressed here that all these assertions remain valid without the assumption (A4).

\item[(iii)] Whether $\lim\limits_{\alpha\to\infty}\lambda_1(\alpha)=\infty$ or not depends on the sign $\hbar_1\ell_1,\,\hbar_2\ell_2$ and the monotonicity of $m$. More precisely, we can conclude the following statements from Theorem \ref{theorem-R} and Theorems \ref{th1.2-n1}-\ref{th1.2-n3}.

\begin{itemize}

\item[(iii-1)] If $\hbar_1\ell_1\geq0,\,\hbar_2\ell_2\geq0$ and $\ell_1,\,\ell_2>0$, then $\lim\limits_{\alpha\to\infty}\lambda_1(\alpha)=\infty$ if and only if
either $m$ increases or decreases on $[0,1]$, or $m$ decreases on $[0,\theta]$ while increases on $[\theta,1]$ for some $\theta\in(0,1)$.

\item[(iii-2)] If either $\hbar_1\ell_1<0,\,\hbar_2\ell_2\geq0$, $\ell_2\not=0$ or $\ell_1=0,\,\ell_2>0$, then $\lim\limits_{\alpha\to\infty}\lambda_1(\alpha)=\infty$ if and only if $m$ increases on $[0,1]$.

\item[(iii-3)] If either $\hbar_1\ell_1\geq0,\,\hbar_2\ell_2<0$, $\ell_1\not=0$ or $\ell_1>0,\,\ell_2=0$, then $\lim\limits_{\alpha\to\infty}\lambda_1(\alpha)=\infty$ if and only if $m$ decreases on $[0,1]$.

\end{itemize}

\end{itemize}

\end{remark}

Before presenting the proof of Theorems \ref{th1.2-n1}-\ref{th1.2-n3}, we would like to demonstrate, through several concrete examples, the impact of the boundary critical points of the advection function $m$ and the boundary conditions on the asymptotic behavior of the principal eigenvalue $\lambda_1(\alpha)$. We will mainly focus on the situation where the degeneracy of $m$ occurs near the left boundary point $x=0$ since similar behaviors happen near the right boundary point $x=1$. For the impact of interior critical points of $m$ on the asymptotic behavior of $\lambda_1(\alpha)$, one may refer to \cite{PZ2018} for some typical examples.

To emphasize the dependence of $\lim_{\alpha\to\infty}\lambda_1(\alpha)$ on the sign of the boundary parameters $\hbar_i,\, \ell_i\,(i=1,2)$ as well as for convenience, instead of $\lim_{\alpha\to\infty}\lambda_1(\alpha)$, below we will use $\lambda_1^{a_i,a_j}(\infty)\,(i=1,2)$  with $a_i\,(i=1,2)$ given by
\begin{equation*}
  a_{i}=+\ \text{ if } \, \hbar_i\ell_i>0 \  \text{or}\  \hbar_i=0;\ \
    a_{i}=0 \ \text{ if }\ \ell_i=0;\ \
     a_{i}=- \ \text{ if }\ \hbar_i\ell_i<0.
 \end{equation*}

Example 1: $m$ is constant on $[0,x_1]$ and increases on $[x_1,1]$; see Figure 1(a).
Let $x_1$ shrink to the boundary point $0$ so that $m$ increases on $[0,1]$; see Figure 1(b).
\begin{figure}[htbp]\label{figure1}
\centering
{\includegraphics[height=1.1in]{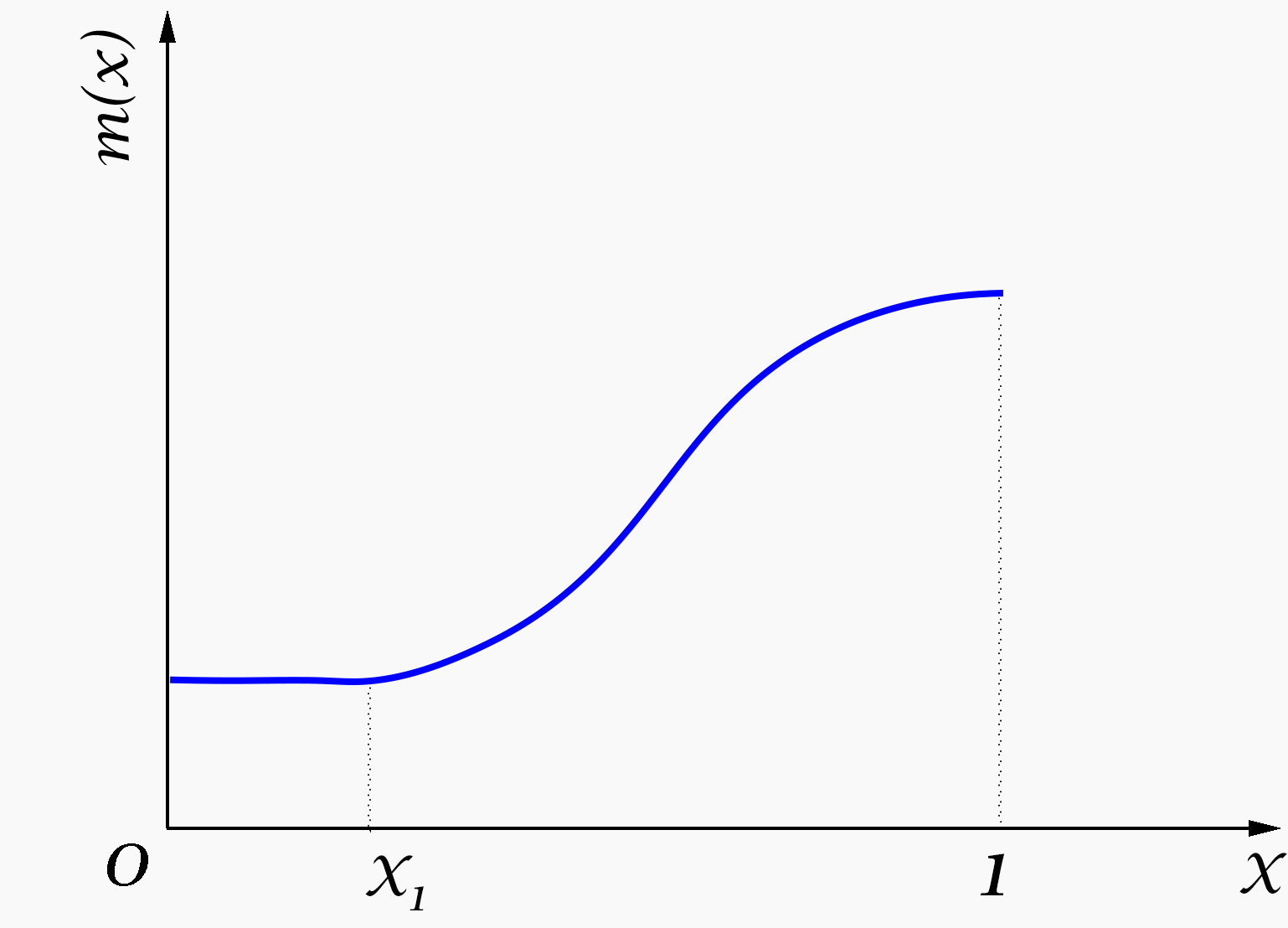}}\ \ \
{\includegraphics[height=1.1in]{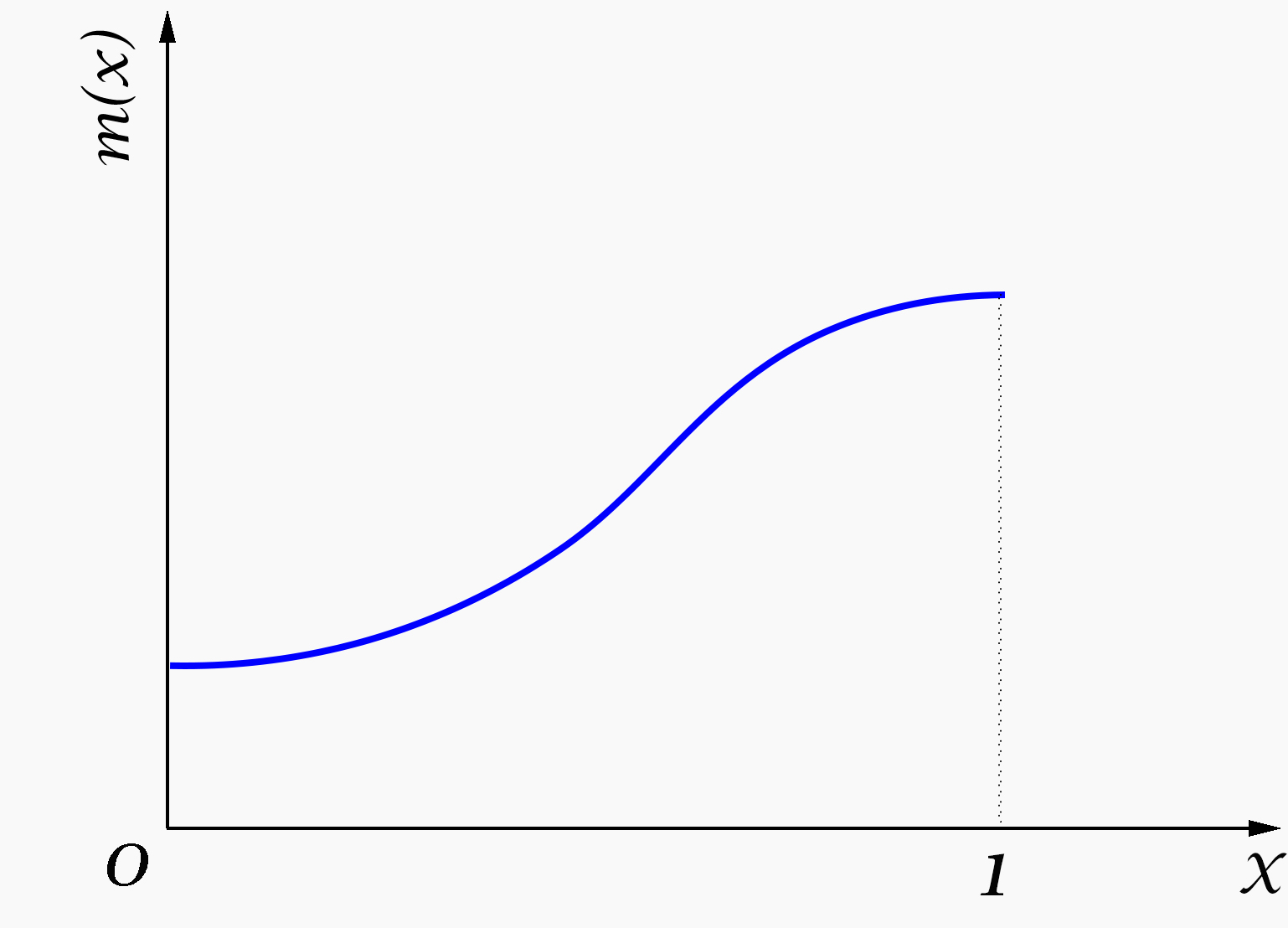}}\ \ \
{\includegraphics[height=1.1in]{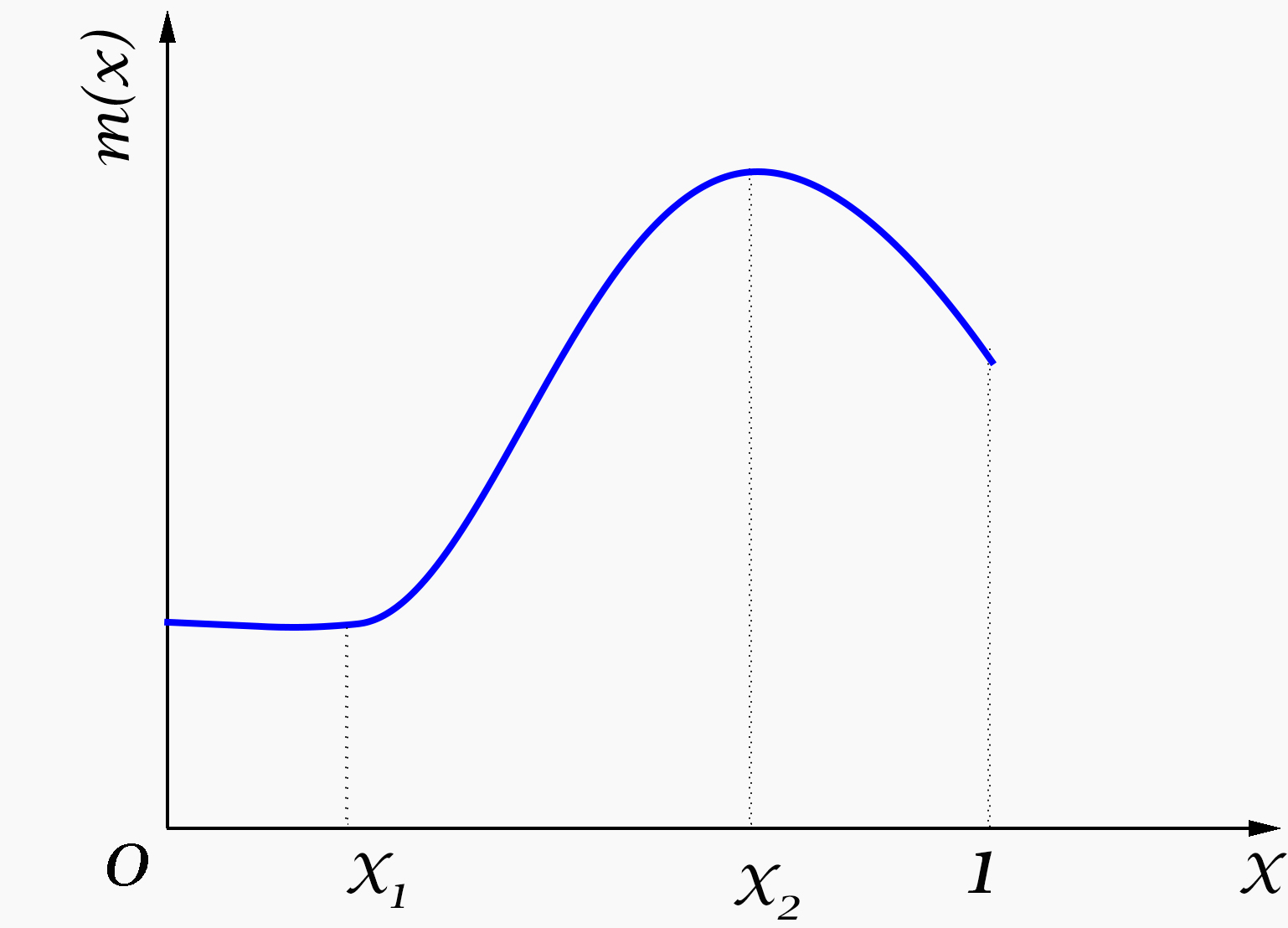}}\ \ \
{\includegraphics[height=1.1in]{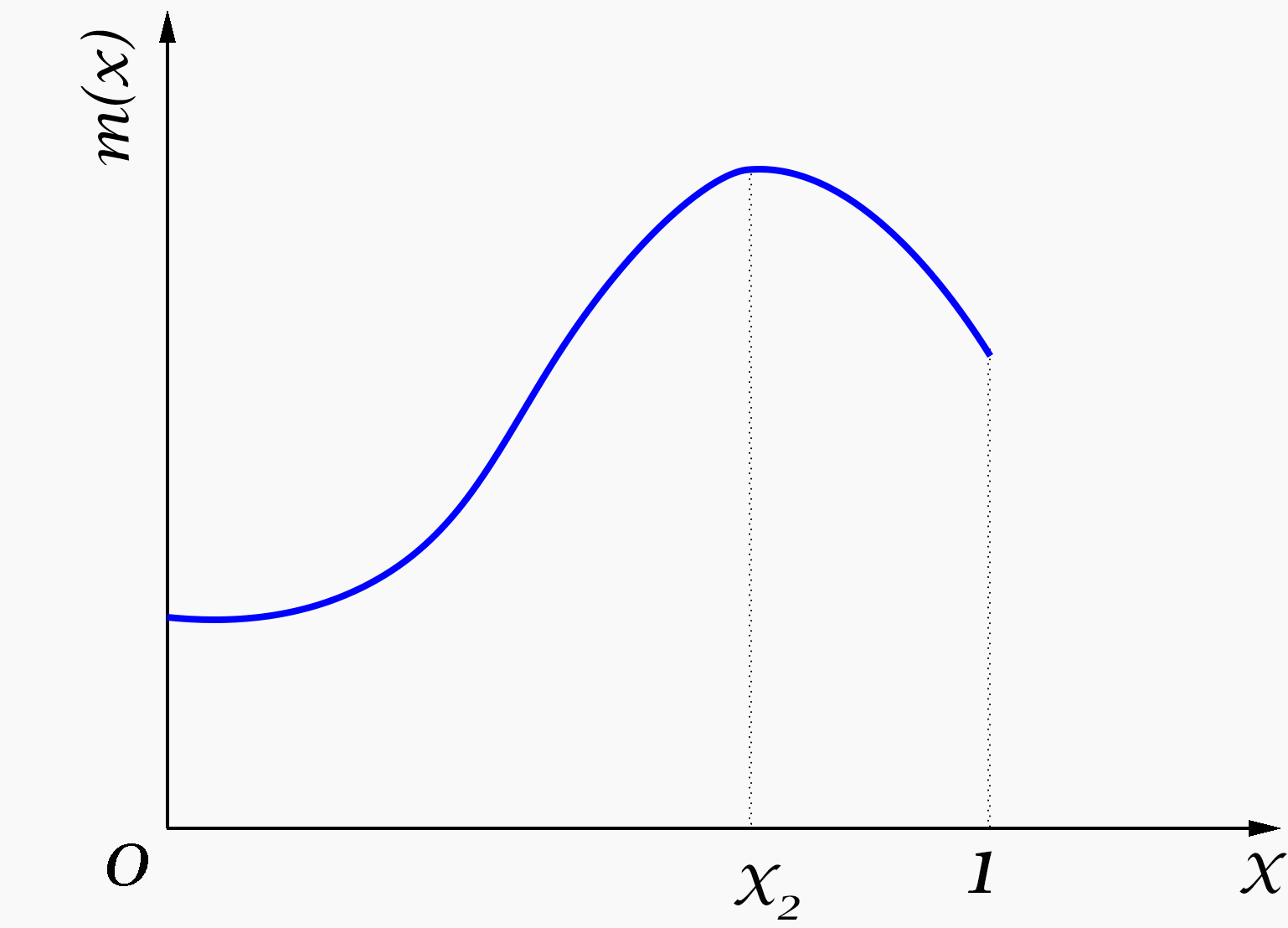}}

{Figure 1(a) \hspace{1.7cm} Figure 1(b)\hspace{1.7cm} Figure 2(a)\hspace{1.7cm} Figure 2(b)}
\end{figure}
According to the results of this section, if $m$ is given as in Figure 1(a), we have the following conclusions:
 \begin{equation}\nonumber
 \left.\begin{array}{llll}
 \lambda_1^{+,+}=\lambda_1^{\mathcal{RD}}(0,x_1);\ &
 \lambda_1^{+,0}=\min\{\lambda_1^{\mathcal{RD}}(0,x_1),\,V(1)\};\ &
 \lambda_1^{0,0}=\min\{\lambda_1^{\mathcal{RD}}(0,x_1),\,V(1)\}; \\
 \lambda_1^{0,-}=-\infty;\ &
 \lambda_1^{-,0}=\min\{\lambda_1^{\mathcal{RD}}(0,x_1),\,V(1)\};&
 \lambda_1^{-,-}=-\infty.
 \end{array}
 \right.
 \end{equation}
It should be pointed out that the values $\lambda_1^{\mathcal{RD}}(0,x_1)$ appearing above are different because they depend on the boundary conditions at $x=0$. Hereafter, we shall use the same notation whenever no confusion is caused.

If $m$ is given as in Figure 1(b), we have
  \begin{equation}\nonumber
\lambda_1^{+,+}=\infty;\ \
\lambda_1^{+,0}=\lambda_1^{0,0}=\lambda_1^{-,0}=V(1);\ \
\lambda_1^{0,-}=\lambda_1^{-,-}=-\infty.
 \end{equation}
Notice that when the point $x_1$ shrinks to the boundary point $0$, all the values $\lambda_1^{\mathcal{RD}}(0,x_1)$ converge to $\infty$. Thus, the asymptotic behavior of $\lim_{\alpha\to\infty}\lambda_1(\alpha)$ changes continuously when $m$ changes continuously from Figure 1(a) to Figure 1(b).

\vskip5pt
Example 2: $m$ is constant on $[0,x_1]$, increases on $[x_1,x_2]$ and decreases on $[x_2,1]$; see Figure 2(a). Let $x_1$ shrink to the boundary point $0$ so that $m$ increases on $[0,x_1]$ and decreases on $[x_2,1]$; see Figure 2(b). If $m$ is given as in Figure 2(a), we have the following conclusions:
 \begin{equation}\nonumber
 \left.\begin{array}{llll}
 \lambda_1^{+,+}=\min\{\lambda_1^{\mathcal{RD}}(0,x_1),\,V(x_2)\};\ &
 \lambda_1^{+,0}=\min\{\lambda_1^{\mathcal{RD}}(0,x_1),\,V(x_2)\};\ &
 \lambda_1^{0,0}=\min\{\lambda_1^{\mathcal{RD}}(0,x_1),\,V(x_2)\}; \\
 \lambda_1^{0,-}=\min\{\lambda_1^{\mathcal{RD}}(0,x_1),\,V(x_2)\};\ &
 \lambda_1^{-,0}=\min\{\lambda_1^{\mathcal{RD}}(0,x_1),\,V(x_2)\};&
 \lambda_1^{-,-}=\min\{\lambda_1^{\mathcal{RD}}(0,x_1),\,V(x_2)\}.
 \end{array}
 \right.
 \end{equation}
If $m$ is given as in Figure 2(b), we have
  \begin{equation}\nonumber
\lambda_1^{+,+}=\lambda_1^{+,0}=\lambda_1^{0,0}=\lambda_1^{-,0}=\lambda_1^{0,-}=\lambda_1^{-,-}=V(x_2).
 \end{equation}
Notice that when the point $x_1$ shrinks to the boundary point $0$, all the values $\lambda_1^{\mathcal{RD}}(0,x_1)$ converge to $\infty$. Again, the asymptotic behavior of $\lim_{\alpha\to\infty}\lambda_1(\alpha)$ changes continuously when $m$ changes continuously from Figure 2(a) to Figure 2(b).

\vskip5pt
Example 3: $m$ is constant on $[0,x_1]$ and decreases on $[x_1,1]$; see Figure 3(a). Let $x_1$ shrink to the boundary point $0$ so that $m$ decreases on $[0,1]$; see Figure 3(b).
\begin{figure}[htbp]\label{figure2}
\centering
{\includegraphics[height=1.1in]{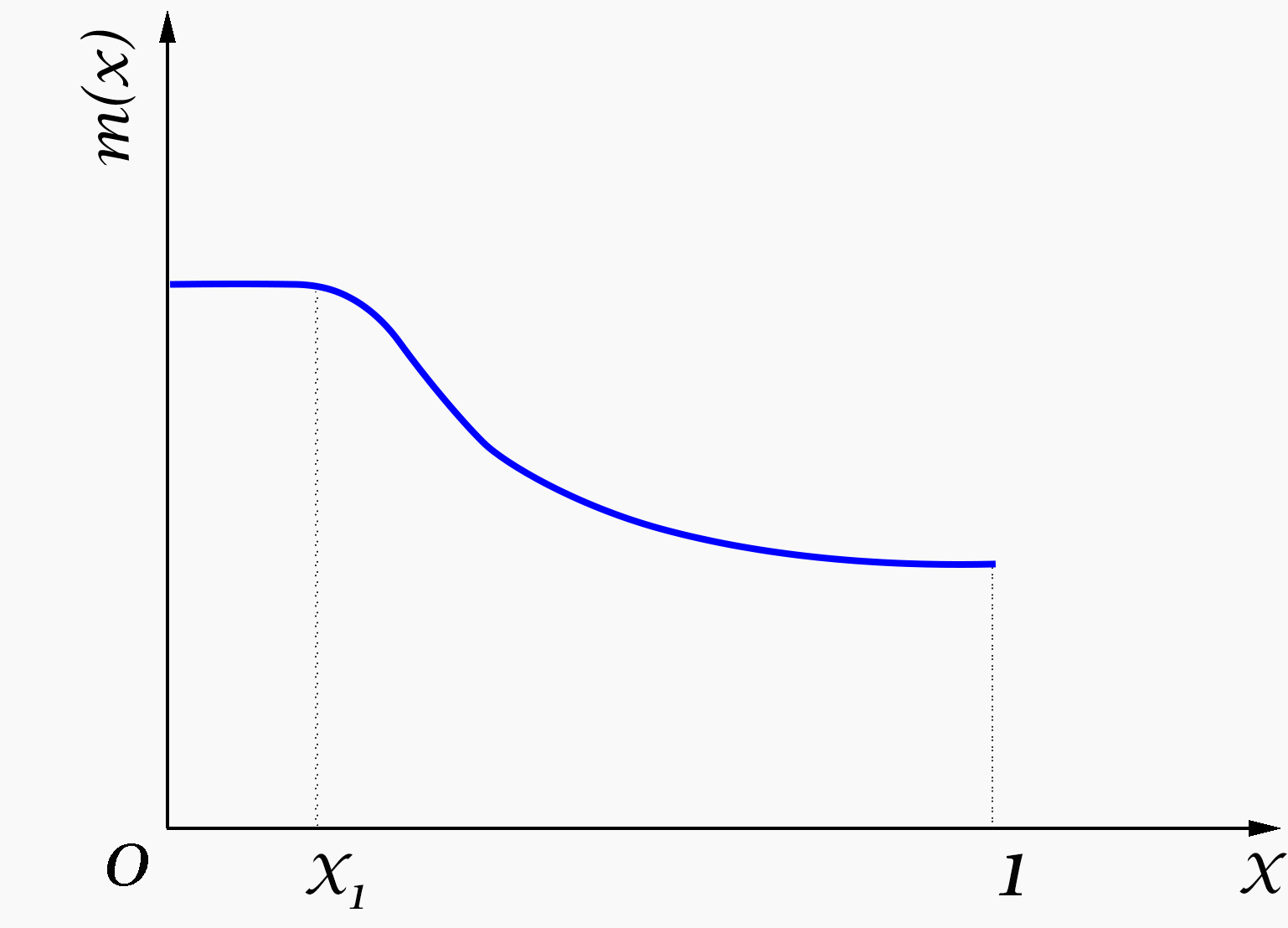}}\ \ \
{\includegraphics[height=1.1in]{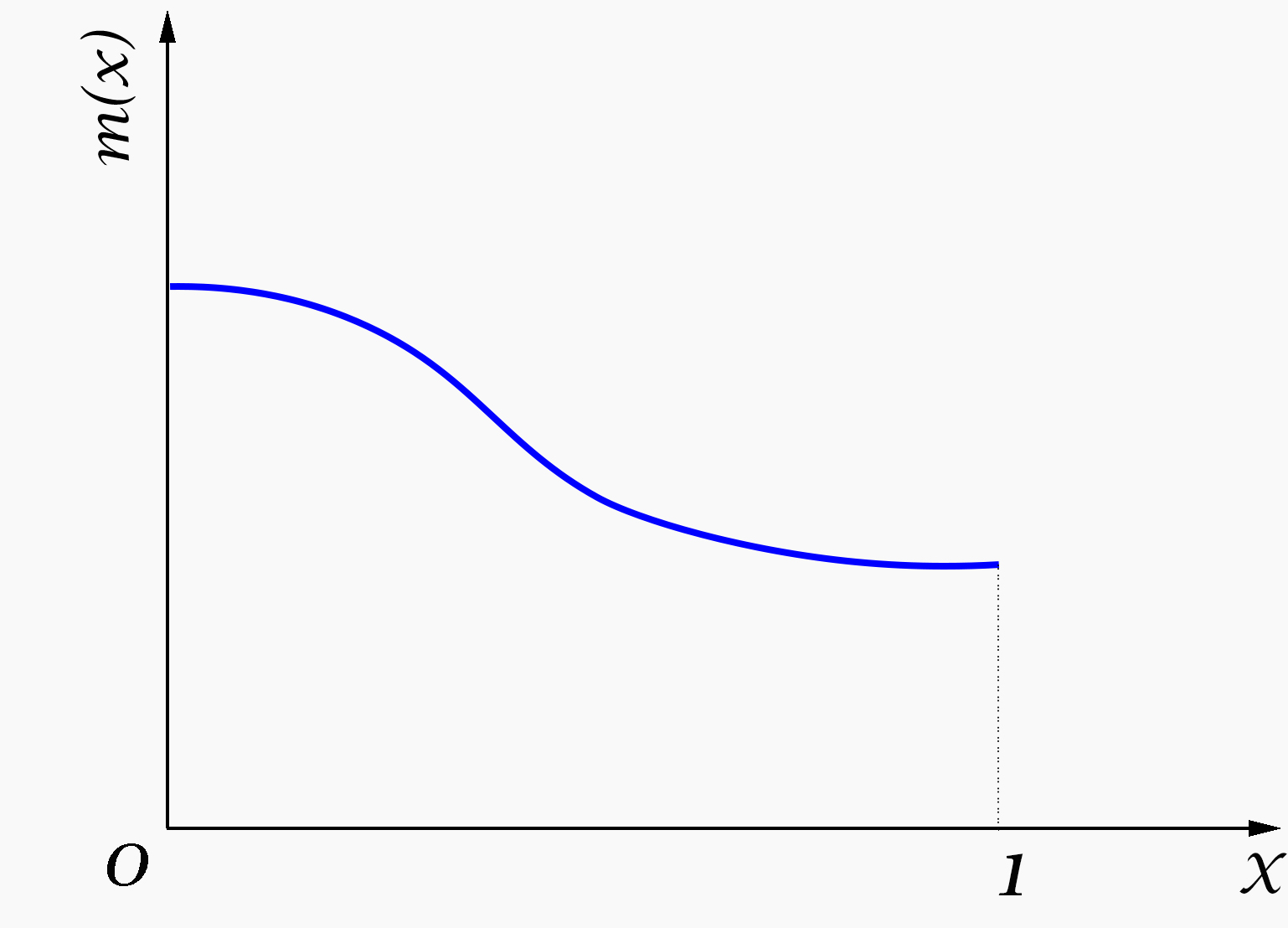}}\ \ \
{\includegraphics[height=1.1in]{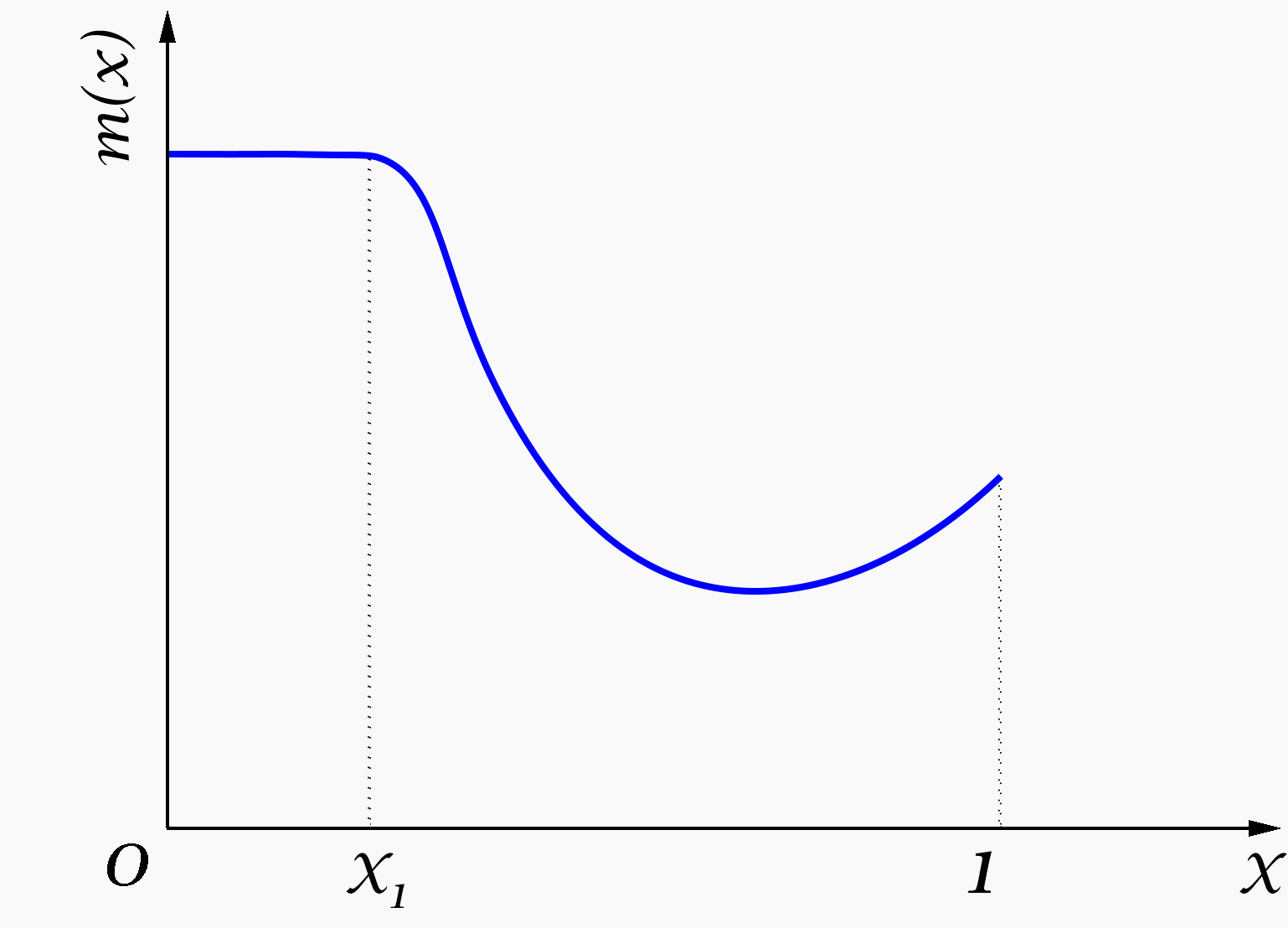}}\ \ \
{\includegraphics[height=1.1in]{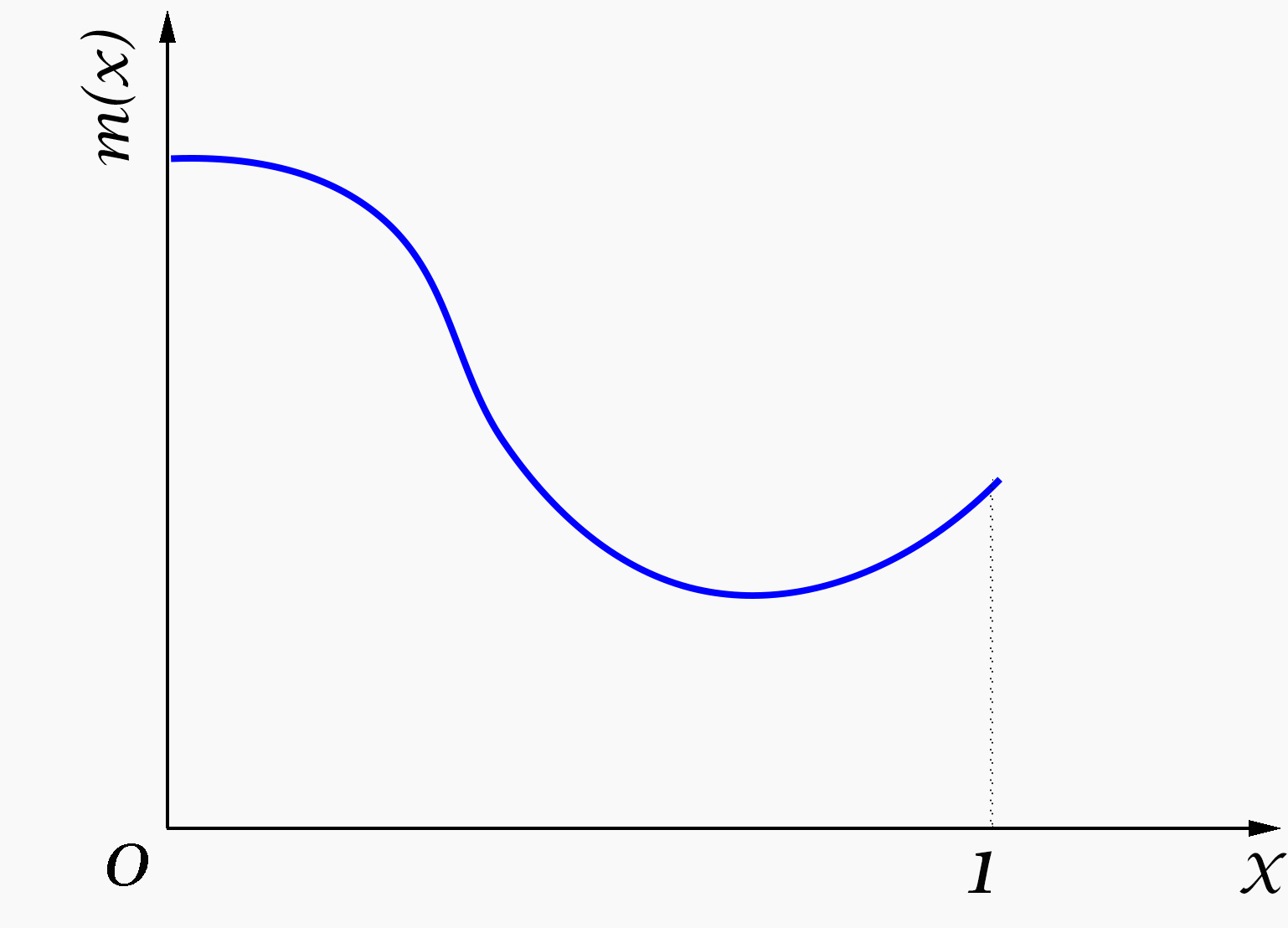}}

{Figure 3(a) \hspace{1.7cm} Figure 3(b)\hspace{1.7cm} Figure 4(a)\hspace{1.7cm} Figure 4(b)}
\end{figure}
If $m$ is given as in Figure 3(a), we have
 \begin{equation}\label{f3-4}
 \left.\begin{array}{llll}
 \lambda_1^{+,+}=\lambda_1^{\mathcal{RD}}(0,x_1);\ &\ \ \
 \lambda_1^{+,0}=\lambda_1^{\mathcal{RD}}(0,x_1);\ &\ \ \
 \lambda_1^{0,0}=\lambda_1^{\mathcal{RD}}(0,x_1); \\
 \lambda_1^{0,-}=\lambda_1^{\mathcal{RD}}(0,x_1);\ &\ \ \
 \lambda_1^{-,0}=\lambda_1^{\mathcal{RD}}(0,x_1);&\ \ \
 \lambda_1^{-,-}=\lambda_1^{\mathcal{RD}}(0,x_1).
 \end{array}
 \right.
 \end{equation}
If $m$ is given as in Figure 3(b), we have
  \begin{equation}\nonumber
\lambda_1^{+,+}=\lambda_1^{+,0}=\infty;\ \
\lambda_1^{0,0}=\lambda_1^{0,-}=V(0);\ \
\lambda_1^{-,0}=\lambda_1^{-,-}=-\infty.
 \end{equation}
Note that in \eqref{f3-4}, as $x_1$ shrinks to the boundary point $0$, it is easily checked that $\lambda_1^{+,+}=\lambda_1^{\mathcal{RD}}(0,x_1)\to\infty,\,\lambda_1^{+,0}=\lambda_1^{\mathcal{RD}}(0,x_1)\to\infty$, $\lambda_1^{0,0}=\lambda_1^{\mathcal{RD}}(0,x_1)\to V(0)$,
$\lambda_1^{0,-}=\lambda_1^{\mathcal{RD}}(0,x_1)\to V(0)$, while
$\lambda_1^{-,0}=\lambda_1^{\mathcal{RD}}(0,x_1)\to-\infty,\,
\lambda_1^{-,-}=\lambda_1^{\mathcal{RD}}(0,x_1)\to-\infty$. Thus, in a certain sense, the asymptotic behavior of $\lim_{\alpha\to\infty}\lambda_1(\alpha)$ changes continuously when $m$ changes continuously from Figure 3(a) to Figure 3(b).

\vskip5pt
Example 4: $m$ is constant on $[0,x_1]$, decreases on $[x_1,x_2]$ for some $x_2\in(x_1,1)$ and increases on $[x_2,1]$; see Figure 4(a). Let $x_1$ shrink to the boundary point $0$ so that $m$ decreases on $[0,x_2]$ and increases on $[x_2,1]$; see Figure 4(b). If $m$ is given as in Figure 4(a), we have
 \begin{equation}\nonumber
 \left.\begin{array}{llll}
 \lambda_1^{+,+}=\lambda_1^{\mathcal{RD}}(0,x_1);\ &
 \lambda_1^{+,0}=\min\{\lambda_1^{\mathcal{RD}}(0,x_1),\,V(1)\};\ &
 \lambda_1^{0,0}=\min\{\lambda_1^{\mathcal{RD}}(0,x_1),\,V(1)\}; \\
 \lambda_1^{0,-}=-\infty;\ &
 \lambda_1^{-,0}=\min\{\lambda_1^{\mathcal{RD}}(0,x_1),\,V(1)\};&
 \lambda_1^{-,-}=-\infty.
 \end{array}
 \right.
 \end{equation}
If $m$ is given as in Figure 4(b), we have
  \begin{equation}\nonumber
\lambda_1^{+,+}=\infty;\ \
\lambda_1^{+,0}=\lambda_1^{0,0}=\lambda_1^{-,0}=V(1);\ \
\lambda_1^{0,-}=\lambda_1^{-,-}=-\infty.
 \end{equation}
 Once again, it is easily seen that in a certain sense, the asymptotic behavior of $\lim_{\alpha\to\infty}\lambda_1(\alpha)$ changes continuously when $m$ changes continuously from Figure 4(a) to Figure 4(b).

\vskip8pt
In what follows, we only give the proof of Theorem \ref{th1.2-n1} since Theorem \ref{th1.2-n2} and Theorem \ref{th1.2-n3} can be verified similarly.

\begin{proof}[Proof of Theorem \ref{th1.2-n1}] As before, through the transformation
$w(x)=e^{\alpha m(x)}\varphi(x)$, problem \eqref{p} reduces to the following one:
 \begin{equation}
 \left\{\begin{array}{ll}
 \medskip
 \displaystyle
 -w''(x)+[\alpha^2(m'(x))^2+\alpha m''(x)+V(x)]w(x)=\lambda_1(\alpha)w(x),\ 0<x<1,\\
 \displaystyle
 (\alpha\hbar_1+\ell_1)w(0)-\hbar_1w'(0)=(-\alpha\hbar_2+\ell_2)w(1)+\hbar_2w'(1)=0.
 \end{array}
 \right.
 \label{4.1d}
  \end{equation}
For each $\alpha$, when $\hbar_2\not=0$, $\lambda_1(\alpha)$ has the variational characterization:
 \begin{equation}
 \left.\begin{array}{lll}
 \medskip
 \displaystyle
 \lambda_1(\alpha)
 &=&
 \displaystyle
 \min_{w\in H^1(0,1),\,\int_0^1 w^2dx=1}\Big\{\int_0^1[(w'-\alpha wm')^2+Vw^2]dx+{{\ell_1}\over{\hbar_1}}w^2(0)+{{\ell_2}\over{\hbar_2}}w^2(1)\Big\},
 \end{array}
 \right.
\label{4.1e}
  \end{equation}
and when $\hbar_2=0$, $\lambda_1(\alpha)$ admits the variational characterization:
 \begin{equation}
 \left.\begin{array}{lll}
 \medskip
 \displaystyle
 \lambda_1(\alpha)
 &=&
 \displaystyle
 \min_{w\in H^1_*(0,1),\, \int_0^1 w^2dx=1}\int_0^1[(w'-\alpha wm')^2+Vw^2]dx+{{\ell_1}\over{\hbar_1}}w^2(0),
 \end{array}
 \right.
\label{4.1f}
  \end{equation}
where $H^1_*(0,1)=\{g\in H^1(0,1):\, g(1)=0\}$.

As before, we shall use $w_\alpha$ with $\int_0^1w_\alpha^2(x)dx=1$ to
denote the principal eigenfunction of \eqref{4.1d}, and assume that, as $\alpha\to\infty$,
$w_\alpha^2$ converges weakly to a unique Radon measure $\mu$ with $\mu([0,1])=1$.

We first assume that $\ell_2\not=0$ and $\mathcal{M}=\mathcal{M}_1=\{1\}$, and claim $\lim\limits_{\alpha\to\infty}\lambda_1(\alpha)=\infty$. Obviously, the condition $\mathcal{M}=\mathcal{M}_1=\{1\}$ implies that the function advection $m$ must be increasing on $[0,1]$, and so $m'(x)\geq0$ for all $x\in[0,1]$.

To show $\lim\limits_{\alpha\to\infty}\lambda_1(\alpha)=\infty$, we proceed indirectly and suppose that there exists a sequence of $\alpha$, labelled by itself for convenience, such that $\lambda_1(\alpha)\leq C$
for some constant $C>0$ and $\alpha\geq0$. Hereafter, the positive constant $C$ may change from place to place but does not depend on $\alpha\geq0$.

As a result, it follows from \eqref{nonreg-1}, \eqref{4.1e} and \eqref{4.1f} that
\begin{equation}\label{4.1g}
\int_0^1[(w_\alpha'-\alpha w_\alpha m')^2+Vw_\alpha^2]dx+{{\ell_1}\over{\hbar_1}}w_\alpha^2(0)\leq C,\ \ \forall \alpha\geq0.
  \end{equation}
On the other hand, from the proof of Theorem \ref{general} (see \eqref{ineq-z4}), we obtain
\begin{equation}\label{4.1h}
-{{\ell_1}\over{\hbar_1}}w_\alpha^2(0)\leq\frac{1}{2}\int_0^1[(w_\alpha'-\alpha w_\alpha m')^2dx+C,\ \ \forall \alpha\geq0.
  \end{equation}
Thus, because of $\ell_1\hbar_1<0$, we infer that
\begin{equation}\label{4.1i}
-{{\ell_1}\over{\hbar_1}}w_\alpha^2(0)\leq C,\ \ \int_0^1[(w_\alpha'-\alpha w_\alpha m')^2dx\leq C,\ \ \forall \alpha\geq0.
  \end{equation}
Combined with \eqref{4.1i} and the fact that $m$ is increasing on $[0,1]$, one can use the same analysis as in the proofs of \cite[Lemma 2.6 and Lemma 2.8]{PZ2018} to conclude that $\mu([0,1))=0$.
Hence, $\mu(\{1\})=1$, which, together with $\int_0^1w_\alpha^2(x)dx=1$ for all $\alpha$, enables us to
conclude that
$$
\limsup_{\alpha\to\infty}\|w_\alpha(\cdot)\|_{L^\infty((0,1))}=\infty.
$$
By virtue of this and \eqref{4.1i}, an argument similar to that of \cite[Lemma 2.7]{PZ2018} can be employed to lead to a contradiction. Therefore, $\lim\limits_{\alpha\to\infty}\lambda_1(\alpha)=\infty$ must hold.

We then claim $\lim\limits_{\alpha\to\infty}\lambda_1(\alpha)=-\infty$ under the assumption $\mathcal{M}=\mathcal{M}_1=\{0\}$. Clearly, $\ell_1\hbar_1<0$, $\mathcal{M}=\mathcal{M}_1=\{0\}$ and \eqref{nonreg-1} imply $\Sigma_{3}=\{0\}$. Thus, Lemma \ref{unbounded} gives $\lim\limits_{\alpha\to\infty}\lambda_1(\alpha)=-\infty$, as desired.

If $\ell_2=0$, it is easily seen from \eqref{4.1e} that $\lambda_1(\alpha)\leq\max_{x\in[0,1]}V(x)$ for all $\alpha\geq0$.

If $\ell_2\not=0$ and $1$ is not an isolated local maximum of $m$, then by (A4), either $m$ is constant on $[1-a_0,1]$ or $m$ is decreasing on $[1-a_0,1]$ for some small $a_0>0$. In view of $\ell_1\hbar_1<0$, the analysis of \cite[Lemma 2.1]{PZ2018} can be used to show that $\limsup\limits_{\alpha\to\infty}\lambda_1(\alpha)<\infty$. This shows that $\lambda_1(\alpha)\leq C$ for all $\alpha\geq0$ when either $\ell_2=0$ or $1$ is not an isolated maximum of $m$. Thus, using \eqref{4.1g} and \eqref{4.1h}, we see that \eqref{4.1i} holds.

On the other hand, if $\mathcal{M}\not=\{0\}$, then by (A4), either $m$ is increasing on $[0,b_0]$ or $m$ is constant on $[0,b_0]$ for some small $b_0>0$. In each case, by Lemma \ref{lowerbound}, we know that
$\liminf\limits_{\alpha\to\infty}\lambda_1(\alpha)>-\infty$.

Consequently, if either $\ell_2=0$ or $1$ is not an isolated maximum of $m$ and $\mathcal{M}\not=\{0\}$, from the above argument we can conclude that
$$
-\infty<\liminf\limits_{\alpha\to\infty}\lambda_1(\alpha)
\leq\limsup\limits_{\alpha\to\infty}\lambda_1(\alpha)<\infty,
$$
and \eqref{4.1i} holds. Then, one can safely follow the argument of \cite[Theorem 1.2]{PZ2018} to establish  Theorem \ref{th1.2-n1}(iii). This completes the proof.
\end{proof}

\vskip15pt

\section*{Declaration of competing interest}
The authors declare that they have no known competing financial interests or personal relationships that could have appeared to influence the work reported in this paper.

\vskip15pt
\section*{Data availability statement}

Data sharing is not applicable to this article as no new data were created or analysed in this study.

\vskip15pt
\section*{Acknowledgments}

R. Peng was supported by NSF of China (No. 12271486, 12171176), and G. Zhang was supported by NSF of China (No. 12171176, 11971187).
The authors would like to thank Dr. Maolin Zhou for his valuable communications during the preparation of this paper. They also wish to express their gratitude to the reviewer for helpful comments, which have improved the presentation of the paper.

\end {document}